\documentclass[11pt]{amsart}
\usepackage{amsmath,amsthm,amssymb,latexsym,graphicx,verbatim,enumerate,%
ifpdf,mdwlist}
\usepackage{mathabx}
\ifpdf
\usepackage{hyperref}
\else
\usepackage[hypertex]{hyperref}
\fi

\usepackage{amssymb, amsmath, amsthm}
\usepackage{color}
\usepackage{graphicx}
\usepackage{cancel}
\usepackage{float}

\DeclareMathOperator{\HA}{HA}

\DeclareMathOperator{\IT}{iT}

\newtheorem{q}{Question}

\newtheorem{thm}{Theorem}
\newtheorem{fact}[thm]{Fact}
\newtheorem{definition}[thm]{Definition}
\newenvironment{defn}{\begin{definition} \rm}{ \end{definition}}
\newtheorem{remark}[thm]{Remark}
\newenvironment{rem}{\begin{remark} \rm}{ \end{remark}}
\newtheorem{lemma}[thm]{Lemma}
\newtheorem{obs}[thm]{Observation}
\newtheorem{example}[thm]{Example}
\newenvironment{eg}{\begin{example} \rm}{ \end{example}}

\newtheorem{cory}[thm]{Corollary}

\newtheorem{claim}[thm]{Claim}

\renewcommand{\phi}{\varphi}

\newcommand{\rel}[1]{\mathrel{#1}}
\newcommand{\ol}[1]{\overline{#1}}

\title[Effective inseparability]
{Effective inseparability, lattices, and pre-ordering relations}%

\author[Andrews]{Uri Andrews}
\email{\href{mailto:andrews@math.wisc.edu }{andrews@math.wisc.edu}}
\urladdr{\url{http://www.math.wisc.edu/~andrews/}}

\address{Department of Mathematics\\
University of Wisconsin\\
Madison, WI 53706-1388\\
USA}

\author[Sorbi]{Andrea Sorbi}
\email{\href{mailto:andrea.sorbi@unisi.it}{andrea.sorbi@unisi.it}}
\urladdr{\url{http://www3.diism.unisi.it/~sorbi/}}
\address{Dipartimento di Ingegneria dell'Informazione e Scienze Matematiche\\
Universit\`a di Siena\\
Siena, 53100\\
Italy}

\thanks{Andrews was partially supported by NSF grant DMS1600228.}

\thanks{Both authors were partially supported by
grant AP05131579 of the Science Committee of the Republic
of Kazakhstan. Sorbi is a member of INDAM-GNSAGA}



\begin{document}
\begin{abstract}
We study effectively inseparable (e.i.) pre-lattices (i.e. structures of
the form $L=\langle \omega, \wedge, \lor, 0, 1, \leq_L\rangle$ where
$\omega$ denotes the set of natural numbers and the following hold:
$\wedge, \lor$ are binary computable operations; $\leq_L$ is a c.e.
pre-ordering relation, with $0 \leq_{L} x \leq_{L} 1$ for every $x$; the
equivalence relation $\equiv_L$ originated by $\leq_L$ is a congruence on
$L$ such that the corresponding quotient structure is a non-trivial bounded
lattice; the $\equiv_L$-equivalence classes of $0$ and $1$ form an
effectively inseparable pair), and show
(Theorem~\ref{thm:ei-lattices-universal}, solving a problem
in~\cite{Montagna-Sorbi:Universal}), that if $L$ is an e.i. pre-lattice
then $\le_{L}$ is universal with respect to all c.e. pre-ordering
relations, i.e. for every c.e. pre-ordering relation $R$ there exists a
computable function $f$ such that, for all $x,y$, $x \rel{R} y$ if and only
if $f(x) \le_{L} f(y)$; in fact (Corollary~\ref{cor:interval}) $\leq_L$ is
locally universal, i.e. for every pair $a<_{L} b$ and every c.e.
pre-ordering relation $R$ one can find a reducing function $f$ from $R$ to
$\le_{L}$ such that the range of $f$ is contained in the interval $\{x: a
\leq_{L} x \leq_{L} b\}$. Also (Theorem~\ref{thm:unif-density}) $\leq_L$ is
uniformly dense, i.e. there exists a computable function $f$ such that for
every $a,b$ if $a<_{L} b$ then $a<_{L} f(a,b) <_{L} b$, and if $a\equiv_{L}
a'$ and $b \equiv_{L} b'$ then $f(a,b)\equiv_{L} f(a',b')$. Some
consequences and applications of these results are discussed: in particular
(Corollary~\ref{cor:sigman}) for $n \ge 1$ the c.e. pre-ordering relation
on $\Sigma_{n}$ sentences yielded by the relation of provable implication
of any c.e. consistent extension of Robinson's $Q$ or $R$ is locally
universal and uniformly dense; and (Corollary~\ref{cor:HA}) the c.e.
pre-ordering relation of provable implication of Heyting Arithmetic is
locally universal and uniformly dense.
\end{abstract}

\maketitle

\section{Introduction}
Given binary relations $R,S$ on the set $\omega$ of natural numbers, we say
that $R$ is \emph{reducible} to $S$ (notation: $R \preceq S$) if there is a
computable function $f$ (called a \emph{computable embedding} of $R$ into
$S$) such that
\[
(\forall x,y)[x  \rel{R} y \Leftrightarrow f(x) \rel{S} f(y)].
\]
Moreover, given a relation $R$ and a class $\mathcal{A}$ of relations on
$\omega$, we say that $R$ is \emph{$\mathcal{A}$-universal} if $R \in
\mathcal{A}$ and $S \preceq R$ for every $S\in \mathcal{A}$.

An interesting case, to which a rapidly growing amount of literature is being
devoted, is when one takes $\mathcal{A}$ to be a class of equivalence
relations in some fixed complexity measure: the notion of universality
becomes in this case a useful tool to measure the computational complexity of
classification problems in computable mathematics (problems which are in fact
equivalence relations): for instance it is shown in
\cite{Fokina-et-al-several} that the isomorphism relation for various
familiar classes of computable groups is $\Sigma^1_1$-complete. An even more
restricted case is when one considers only computably enumerable (henceforth
abbreviated as c.e.) equivalence relations (c.e. equivalence relations  will
be abbreviated as \emph{ceers} throughout the paper), see e.g. \cite{ceers}.
Interest in ceers is motivated by their importance in logic and algebra: for
instance, modulo suitable G\"odel numberings identifying the various objects
with numbers, the relation of provable equivalence between formulas of a
formal system, and word problems for finitely presented groups and
semigroups, are ceers. When studying ceers, particular attention has been
given to $\Sigma^0_1$-universal ceers: see for instance
\cite{andrews2017survey} for a survey.

It is very natural to move from equivalence relations to the wider class
consisting of the pre-ordering relations. Indeed, many interesting
classification problems in computable mathematics occur as pre-ordering
relations (for instance embedding problems, instead of isomorphism problems,
for various classes of computable groups). Restriction to c.e. pre-ordering
relations (i.e. to pre-orderings $R$ such that the set $\{\langle x,y
\rangle: x \rel{R} y\}$ is c.e.) is again mostly motivated by their
importance in logic: for instance if one considers a strong enough formal
system, then the relation of provable implication between formulas of a
formal system is a $\Sigma^0_1$-universal pre-ordering relation by
\cite{Montagna-Sorbi:Universal}, i.e. it is c.e. and every c.e. pre-ordering
relation is reducible to it.

In this paper we investigate some classes of $\Sigma^0_1$-universal
(henceforth called simply \emph{universal}) pre-ordering relations.
Our quest for interesting universal c.e. pre-ordering relations is based on
the notion of effective inseparability for lattices (or, better,
pre-lattices, as we will call them in this paper). In this regard, this paper
can be viewed as a generalization of some aspects of the work done by
Montagna and Sorbi~\cite{Montagna-Sorbi:Universal},
Nies~\cite{Nies:Effectively} and Shavrukov~\cite{Shavrukov}, concerning
effectively inseparable Boolean algebras.

To make more sense of this, we first need some definitions. In general, given
a type $\tau=\langle \{f_{i}\}_{i \in I}, \{R_{j}\}_{j \in J}\rangle$ for
operations and relations, then a \emph{pre-structure of type $\tau$} is a
structure $A= \langle \omega, \{f^A_{i}\}_{i \in I}, \{R^A_{j}\}_{j \in I},
\leq_A \rangle$ with operations and relations interpreting the type, such
that $\leq_A$ is a pre-ordering and the equivalence relation $\equiv_A$
induced by $\leq_A$ (i.e.\ $x \equiv_{A} y$ if $x \leq_{A} y$ and $y \leq_{A}
y$) is a congruence for the operations and the relations. If in addition
operations and relations are finitary, $\leq_A$ is c.e., each operation
$f^A_i$ (interpreting $f_i$) is computable uniformly in $i$, and each
relation $R^A_j$ (interpreting $R_j$) is c.e. uniformly in $i$, then we say
that the pre-structure is a \emph{c.e. pre-structure of type $\tau$}.

\begin{defn}\label{defn:associated-quotient}
If $A$ is a pre-structure then the quotient $A_{/\equiv_A}$ will be called
the \emph{quotient structure associated with $A$}. (Notice that
$A_{/\equiv_A}$ is partially ordered  by the partial ordering originated by
$\leq_A$.)
\end{defn}

For instance, a pre-ordering relation can be just viewed as a pre-structure
$R=\langle \omega, \leq_R \rangle$ (called a \emph{pre-order}) with empty
type, whose associated quotient structure is a partial order. It will be
customary in the following to identify pre-ordering relations with their
corresponding pre-orders, so that the terms ``pre-ordering relation'' and
``pre-order'' will be treated as synonymous.

\begin{defn}
A \emph{c.e. pre-lattice} is a c.e. pre-structure $L=\langle \omega, \wedge,
\vee, \leq_{L}\rangle$ such that the associated quotient structure is a
lattice (with $\wedge$ and $\lor$ inducing in the quotient the operations of
meet and join, respectively).
\end{defn}

Having in mind this definition, it is now clear what \emph{c.e. Boolean
pre-algebras}, \emph{c.e. distributive pre-lattices}, \emph{c.e. Heyting
pre-algebras}, \emph{c.e. pre-semilattices}, etc., are. In particular a
\emph{c.e. bounded pre-lattice} is a c.e. pre-structure $L=\langle \omega,
\wedge, \lor, 0,1, \leq_L \rangle$ such that the associated quotient
structure is a bounded lattice for which the $\equiv_{L}$-equivalence classes
$[0]_L$ and $[1]_L$, of $0$ and $1$, respectively, are the least element and
the greatest element. (It is clear that up to a computable permutation of
$\omega$ we may assume that in a c.e. bounded pre-lattice $L$ one has that
$0\leq_{L} x \leq_{L} 1$ for every $x$, so that our choice of $0$ and $1$ as
representatives of the least element and the greatest element, respectively,
is no loss of generality.)

\begin{defn}\label{defn:isomorphism}
By a \emph{computable isomorphism} between pre-structures $A,B$ we mean a
computable function $f$ which reduces $\leq_A \preceq \leq_B$, and the
assignment on equivalence classes $[x]_{A} \mapsto [f(x)]_{B}$ is a well
defined isomorphism between the corresponding quotient structures of $A$ and
$B$.
\end{defn}

\begin{rem}
Our terminology follows largely \cite{Montagna-Sorbi:Universal} except for
the fact that what we call here ``c.e. pre-structures'' are therein called
\emph{positive pre-structures}. In the Russian literature, a \emph{positively
numbered structure} is a pair $\langle A^-, \nu\rangle$, where $A^-$ has
universe $\omega$ with computable operations and c.e. relations such that
$\nu: \omega \rightarrow A$ is a surjective function and the relation
``$\nu(x)=\nu(y)$'' is a c.e. equivalence relation inducing a congruence
$\theta$ on $A^-$. So, what we call in this paper ``c.e. pre-lattices'',
``c.e. Boolean pre-algebras'' etc., in the Russian literature would be
certain ``positively numbered lattices'', ``positively numbered Boolean
algebras'', etc.).  Our ``e.i. Boolean pre-algebras'' are e.i. Boolean
algebras in the sense of \cite{Nerode-Remmel:Survey, Shavrukov}. See also
\cite{Selivanov03,Gavryushkin-Jain-Khoussainov-Stephan,
Franketal:structures,FoKhSeTu} for an approach to c.e structures similar to
the one taken in this paper.
\end{rem}

\begin{rem}\label{rem:more-general}
One could give a seemingly more general definition of a c.e. pre-structure of
type $\tau$, by asking that the universe of the structure be just a c.e. set
instead of $\omega$. Up to computable isomorphisms, nothing is however gained
in generality by doing so. Let us briefly give a justification of this claim
in the case of c.e. pre-lattices. Let  $L=\langle U \wedge, \lor, \leq
\rangle$ be a c.e. pre-lattice, where $U$ is a c.e. set. Let $h: \omega
\longrightarrow U $ be a computable surjection. This induces a c.e.
pre-lattice $L'=\langle \omega, \wedge', \lor', \leq_{L'}\rangle$ where
$x\le_{L'} y$ if and only if $h(x) \le_{L} h(y)$, and (letting $h^{-1}$ be
any partial computable function such that $h(h^{-1}(x))=x$ for every $x \in
U$) $x \wedge' y=h^{-1}(h(x) \wedge h(y))$, $x \lor' y=h^{-1}(h(x) \lor
h(y))$. The assignment on equivalence classes $[x]_{A} \mapsto [h(x)]_{B}$ is
clearly a well-defined isomorphism between the corresponding quotient
structures of $A$ and $B$. Notice that if $U$ is infinite, then we can take
$h$ to be a computable permutation of $\omega$. On the other hand, if the
equivalence classes of $\equiv_L$ are infinite then from any computable
isomorphism we can effectively go to a computable isomorphism induced by a
computable permutation.
\end{rem}

The following notion from computability theory will play a fundamental role
in this paper. We recall that a disjoint pair $(A,B)$ of sets of numbers is
called an \emph{effectively inseparable} (or simply, \emph{e.i.}) pair if
there exists a partial computable function $\phi(u,v)$ (called a
\emph{productive function} for the pair)  such that
\[
(\forall u,v)[A \subseteq W_u \,\&\, B \subseteq W_v
 \,\&\, W_u \cap W_v=\emptyset \Rightarrow \phi(u,v)\downarrow
 \notin W_u\cup W_v].
 \]
It is known, see e.g. \cite[p.44]{Soare:Book}, that every e.i.  pair of c.e.
sets has a total productive function: in fact, one can uniformly go from any
index of productive function to an index of a total productive function. Also
if $(A_0,A_1)$, $(X_0,X_1)$ are disjoint pairs of c.e. sets, with $(A_0,A_1)$
e.i. and either $(A_0,B_1) \subseteq (B_0,B_1)$ (meaning $A_0\subseteq B_0$
and $A_1 \subseteq B_1$) or $(A_0,B_1) \leq_m (B_0,B_1)$ (meaning that there
is a computable function simultaneously $m$-reducing $A_0\leq_m B_0$ and
$A_1\leq_m B_1$) then $(B_0,B_1)$ is e.i. as well.

\begin{defn}\label{def:e.i.-pre-lattice}
An \emph{effectively inseparable} (or \emph{e.i.}) \emph{pre-lattice} is a
c.e. bounded pre-lattice $L= \langle \omega, \wedge, \lor, 0,1 \le_{L}
\rangle$ such that the pair of equivalence classes $([0]_L, [1]_L)$ is
e.i.\,. (The definition of course specializes to c.e. Boolean pre-algebras
(defining \emph{e.i. Boolean pre-algebras}) and c.e. Heyting pre-algebras
(defining \emph{e.i. Heyting pre-algebras}), which are examples of c.e.
bounded pre-lattices.)
\end{defn}

It is shown in \cite{Montagna-Sorbi:Universal} (based on
\cite{Pour-El-Kripke}; see also \cite{Ianovski-et-al}) that if $B$ is an e.i.
Boolean pre-algebra, and $\le_{B}$ is the c.e. pre-ordering relation relative
to $B$, then $\le_{B}$ is universal. Ianovski~et~al.~\cite{Ianovski-et-al}
have extended this result throughout the arithmetical hierarchy by showing
that for every $n\ge 1$, the pre-ordering relation of a
$\Sigma^{0}_{n}$-effectively inseparable Boolean pre-algebra (called
\emph{$\Sigma^{0}_{n}$-effectively inseparable Boolean algebra}) is
a $\Sigma^{0}_{n}$-universal pre-ordering relation.

Moving in this paper from e.i. Boolean pre-algebras to e.i. pre-lattices, we
show that the pre-ordering relation of any e.i. pre-lattice is universal. In
fact, we show that every c.e. pre-order can be computably embedded in any
nonempty interval of any e.i. pre-lattice (this property will be called
\emph{local universality}: see Corollary~\ref{cor:interval}). Moreover every
e.i. pre-lattice $L$ is uniformly dense (Theorem~\ref{thm:unif-density}),
meaning that there is a computable function $f(a,b)$ such that for every
$a,b$ if $a<_{L} b$ then $a<_{L} f(a,b) <_{L} b$, and if $a\equiv_{L} a'$ and
$b \equiv_{L} b'$ then $f(a,b)\equiv_{L} f(a',b')$. Remarkably, these results
on universality, local universality and uniform density, do not use
distributivity (on the other hand Example~\ref{ex:nondistr} shows that  e.i.
non-distributive pre-lattices do exist), but do not hold in general if we
move from e.i. pre-lattices to e.i. pre-semilattices, i.e. c.e. bounded
pre-semilattices $S=\langle \omega, \lor, 0,1, \le_S\rangle$ such that the
pair $([0]_S, [1]_S)$ is e.i. (in Observation~\ref{obs:not-for-semilattices}
we exhibit an e.i. upper pre-semilattice whose pre-ordering relation is not
universal). We also point out (Theorem~\ref{thm:several} and
Theorem~\ref{thm:ei-prelattices-not-isomorphic}) several distinct natural
computable isomorphism types of e.i. pre-lattices.

The greater generality of considering pre-lattices instead of Boolean
pre-algebras enables us to extend to consistent extensions of very weak
systems of arithmetic such as Robinson's systems $Q$ or $R$ some of the
results of Shavrukov and Visser~\cite{Shavrukov-Visser} stated for consistent
extensions of Elementary Arithmetic $EA$ (see Corollary~\ref{cor:sigman}): in
particular (Corollary~\ref{cor:sigman}) for $n \ge 1$ the c.e. pre-ordering
relation on $\Sigma_{n}$ sentences yielded by the relation of provable
implication of any c.e. consistent extension of Robinson's $Q$ or $R$ is
locally universal and uniformly dense.

Since our results apply to every e.i. lattice, they hold of e.i. Heyting
pre-algebras as well (for instance, the Heyting pre-algebra of the Lindenbaum
sentence algebra of any consistent intuitionistic extension of Heyting
Arithmetic, as defined in Definition~\ref{defn:preHalgebras}) to which one
can therefore extend the above mentioned results showing for instance
(Corollary~\ref{cor:HA}) that the c.e. pre-ordering relation of provable
implication of Heyting Arithmetic is locally universal and uniformly dense.

In contrast with what happens for e.i. Boolean pre-algebras which are
universal with respect to the class of all c.e. Boolean pre-algebras (a c.e.
Boolean pre-algebra $B$ is \emph{universal} with respect to  the class of all
c.e. Boolean pre-algebras if for every c.e. Boolean pre-algebra $B'$ one has
that $\leq_{B'} \preceq \leq_{B}$ via a reduction $f$ which induces a
monomorphism between the associated quotient structures) we show that,
although guaranteeing universality of its pre-ordering relation, being an
e.i. pre-lattice does not suffice  to guarantee universality with respect to
the class of all c.e. pre-lattices: Theorem~\ref{thm:not-all-of-them}
provides an e.i. distributive pre-lattice such that not all c.e. distributive
pre-lattices reduce to it via a reduction inducing a monomorphisms between
the corresponding quotient structures.

\subsection{Further notations and terminology}\label{ssct:further}
If $R=\langle \omega, \leq_{R} \rangle$ is a pre-order then we use the
following notations. We write $x <_{R} y$ for $x \leq_{R} y$ but $y
\cancel{\leq_{R}} x$; the symbol $|_{R}$ denotes $\leq_{R}$-incomparability,
i.e. $x \rel{|_{R}} y$ if $x \cancel{\leq_{R}} y$ and $y \cancel{\leq_{R}}
x$; if $a \in \omega$ and $X \subseteq \omega$ then we write $a \leq_{R} X$
if $a \leq_{R} x$, for all $x \in X$ (similarly define $a <_{R} X$, $X
\leq_{R} a$, and $X  <_{R} a$). We also use interval notation, i.e.
$[x,y]_{\leq_{R}}=\{a: x \leq_{R} a \leq_{R} y\}$, and other possible
interval variations such as for instance $(x,y)_{\leq_{R}}$.

In some of the proofs, we will make use of infinite computable lists of fixed
points as given by the Recursion Theorem. The use of these infinite lists can
be formally justified as follows: we fix a single index $j$ so that we
control $\phi_j$ by the Recursion Theorem, and we then take a computable list
$(j_i)_{i\in\omega}$ of indices for the columns $\phi_{j_i}(k)=\phi_j(\langle
i,k\rangle)$. We can then control the $\phi_{j_i}$ in any order we wish, as
we are simply controlling the single function $\phi_j$. An alternative formal
justification to arguments using infinite lists of fixed points is also
provided by the Case Functional Recursion Theorem~\cite{Case:Functional}: see
also \cite{Odifreddi:BookII} for useful comments about this theorem.

All pre-orders considered in this paper are assumed to be non-trivial, i.e.
the associated quotient structures do not collapse to a single element.

\section{Ceers, pre-orders and effective inseparability}
From the theory of ceers we recall the following definitions (see e.g.
\cite{ceers}). Particularly important for us is ``uniform finite
precompleteness'', a notion due to Montagna~\cite{Montagna:ufp}, and widely
used in \cite{Bernardi-Montagna:extensional,Shavrukov:ufp}.

\begin{defn}\label{def:ufp-etc}
Let $R$ be a non-trivial ceer (i.e. $R$ is a ceer with at least two
equivalence classes):
\begin{itemize}

\item $R$ is \emph{effectively inseparable} (abbreviated as \emph{e.i.}) if
    every pair of distinct equivalence classes is e.i.;

\item $R$ is \emph{uniformly effectively inseparable} (abbreviated as
    \emph{u.e.i.}) if there exists a partial computable function
    $\chi(x,y,u,v)$ such that if $x \cancel{\rel{R}} y$ then the partial
    computable function $\psi(u,v)=\chi(x,y,u,v)$ is a productive function
    for the pair of equivalence classes $([x]_{R}, [y]_{R})$;

\item $R$ is \emph{precomplete} if there exists a computable function
    $f(e,x)$ such that for all $e,x$ if $\phi_{e}(x)\downarrow$ then
    $f(e,x)\rel{R} \phi_{e}(x)$;

\item $R$ is \emph{uniformly finitely precomplete} (abbreviated as
    \emph{u.f.p.})  if there exists a computable function $k(D,e,x)$
    (called a \emph{totalizer}, where $D$ is a finite set given by its
    canonical index) such that for all $D, e,x$ if $\phi_{e}(x)$ converges
    to a number $R$-equivalent to some $d \in D$ then $k(D,e,x) \rel{R}
    \phi_{e}(x)$.
\end{itemize}
\end{defn}

\begin{rem}\label{rem:uniformity-for-uei}
Notice that since for pairs of c.e. sets it is possible to uniformly go from
productive functions to total productive functions (see
e.g.~\cite[p.44]{Soare:Book}), it is easy to see that it is possible for
ceers to uniformly go from uniform productive functions witnessing
u.e.i.-ness to total uniform productive functions witnessing u.e.i.-ness.
Thus we can assume in the following that our productive functions witnessing
u.e.i.-ness for ceers are in fact total.
\end{rem}

\begin{fact}\label{fact:implications}
For ceers we have: $\textrm{precomplete } \Rightarrow \textrm{u.f.p. }
\Rightarrow \textrm{u.e.i. } \Rightarrow \textrm{e.i.}$. It is not known
whether the implication  $\textrm{u.e.i. } \Rightarrow \textrm{u.f.p.}$
holds; the other implications are strict. Every u.e.i. ceer $R$ is
$\Sigma^0_1$-universal as an equivalence relation, i.e. every ceer can be
reduced to it.
\end{fact}

\begin{proof}
For a proof see \cite{ceers}. The fact that precomplete ceers are u.e.i. was
first noticed in \cite{Visser:Numerations}.
\end{proof}

Via the equivalence relation induced by a pre-ordering relation we export
Definition~\ref{def:ufp-etc} to pre-orders and pre-structures, as
follows.

\begin{defn}\label{defn:ei-preorders-and-other-things}
A c.e. pre-order $R=\langle \omega, \leq_R\rangle$ is called
\emph{effectively inseparable}, (respectively: \emph{uniformly effectively
inseparable}, \emph{precomplete}, \emph{uniformly finitely precomplete}), if
so is the associated ceer $\equiv_R$.
\end{defn}

\begin{rem}
A warning on the terminology introduced in the previous definition: to be
meticulous, one could find the definition of an e.i. pre-lattice not coherent
with Definition~\ref{defn:ei-preorders-and-other-things}, for which an e.i.
pre-order is a c.e. pre-order in which all pairs of distinct equivalence
classes are e.i., whereas for e.i. pre-lattices we only require, following
\cite{Montagna-Sorbi:Universal}, effective inseparability of the equivalence
classes of $0,1$. It will turn out in any case
(Theorem~\ref{thm:from-ei-to-ufp}) that e.i. pre-lattices are also e.i.
pre-orders in the sense of
Definition~\ref{defn:ei-preorders-and-other-things}, in fact they are
u.e.i.\,.
\end{rem}

The following observation shows a simple application of effective
inseparability to c.e. pre-orders and c.e. pre-lattices.

\begin{obs}\label{obs:incomparability}
If $R=\langle \omega, \leq_{R} \rangle$ is a c.e. pre-order so that
$([u]_{R}, [v]_{R})$, where $u<_{R} v$, is an e.i. pair of sets then for
every non-empty c.e. set $X$ one can find (uniformly from $u,v$ and a c.e.
index of $X$) an element $y$ such that if $u<_{R}X<_{R} v$ then $y |_{R} X$.
If in addition $R$ is a c.e. pre-lattice then we can find $y$ as above such
that if $u<_{R}X<_{R} v$ then $y |_{R} X$ and $u<_R y <_R v$.
\end{obs}

\begin{proof}
Assume that $u,v, X$ are as in the statement of the observation. Let $p$ be a
productive function for the pair $([u]_{R}, [v]_{R})$ of e.i. sets. Let
\[
U=\{y: (\exists x \in X)[y \leq_{R} x]\}  \quad \text{and} \quad
V=\{y: (\exists x\in X)[y \geq_{R} x]\},
\]
(thus $u\in U$ and $v\in V$) and by the Reduction Principle (see
\cite{Rogers:Book}) let $U'\subseteq U$, $V'\subseteq V$ be c.e. sets such
that $U' \cap V'=\emptyset$ and $U\cup V= U' \cup V'$. If $u<_{R}X<_{R} v$
then clearly $[u]_{R} \subseteq U'$ and $[v]_{R}\subseteq V'$. If $u',v'$ are
c.e. indices of $U', V'$ then $p(u',v') \notin U\cup V$ and thus
$p(u',v')|_{R} x$ with all $x\in X$. To show uniformity, use the fact that
$u',v'$ can be uniformly found starting from $u,v$, and a c.e. index of $X$.

If instead of just a c.e. pre-order we start with a c.e. pre-lattice
$R=\langle \omega, \wedge, \vee, \leq_{R}\rangle$ then define
\[
U=\{y: (\exists x \in X)[(u \lor y) \wedge v \leq_{R} x]\}
\quad \text{and} \quad
V=\{y: (\exists x\in X)[(u \lor y) \wedge v \geq_{R} x]\}.
\]
Arguing as above, by the Reduction Principle get from $U,V$ disjoint c.e.
sets $U', V'$ with indices $u',v'$ such that $U\cup V = U' \cup V'$: then it
is easy to see that $y=(u \lor p(u',v')) \wedge v$ satisfies the claim.
\end{proof}

\section{Examples of e.i. pre-lattices}
In this section we point out several examples of e.i. pre-lattices, many of
them coming from the already existing literature.

\begin{defn}
If $S$ is a c.e. set $S$ of formulas of a first order language $L$, by a
\emph{G\"odel numbering of $S$} we will mean \emph{any} computable $1$-$1$
function from $S$ to $\omega$.
\end{defn}

\subsection{E.i. Boolean pre-algebras}\label{sssct:eiBoolean}
E.i. Boolean pre-algebras appear very naturally in logic.

\begin{eg}\label{ex:BT}
Consider any c.e. consistent extension $T$ of Robinson's systems $R$ or $Q$
for first order arithmetic (as presented for instance in
Definitions~III.6.10~and~III.8.2 of \cite{Smorynski-logical}), with the usual
first order language $L$ for arithmetic. Our choice of $R$ is motivated by
the fact that we would like to be as general as possible, and indeed $R$ is
often recognized as a prototypical example of a weak theory of arithmetic,
and although $Q$ interprets $R$ (see e.g.
\cite[Theorem~III.8.3]{Smorynski-logical}) we explicitly mention $Q$ as well
as an example of a weak finitely axiomatizable theory.

Let us fix a G\"odel numbering $\gamma$ of $\textrm{Sent}(L)$ (all sentences
of $L$) \emph{onto} $\omega$ so that $\gamma(\neg 0=0)=0$ and
$\gamma(0=0)=1$.
\begin{defn}
The \emph{Boolean pre-algebra of $T$ relatively to $\gamma$} is the Boolean
pre-algebra
\[
B_T^\gamma=\langle \omega, \wedge^\gamma, \lor^\gamma, \neg^\gamma, 0,1,
\leq_T^\gamma \rangle
\]
where $\wedge^\gamma, \lor^\gamma, \neg^\gamma$ are (via $\gamma$) the usual
propositional connectives (i.e. $x\wedge^\gamma y=\gamma(\gamma^{-1}(x)\wedge
\gamma^{-1}(x))$, etc.), and $x\leq_T^\gamma y$ if $\vdash_T \gamma^{-1}(x)
\rightarrow \gamma^{-1}(y)$.
\end{defn}

\begin{rem}\label{rem:independence}
It is easy to see that up to computable isomorphisms induced by computable
permutations of $\omega$, the Boolean pre-algebra $B_{T}^\gamma$ does not
depend on the chosen $\gamma$, so it is fair to refer to $B_T^\gamma$ as just
\[
B_T=\langle \omega, \wedge, \lor, \neg, 0,1,\leq_T \rangle,
\]
without mentioning, unless when really needed, the chosen G\"odel numbering.
Notice that $B_T$ is not trivial by consistency of $T$.
\end{rem}

\begin{thm}\label{thm:BT}
$B_T$ is an e.i. Boolean pre-algebra.
\end{thm}

\begin{proof}
This immediately follows from the well known fact (see for
instance~\cite{Smorynski-logical}) that the sets of theorems and
anti-theorems of $T$ (where $\sigma$ is an \emph{anti-theorem} if $\neg
\sigma$ is a theorem) form an e.i. pair of c.e. sets. We give however a
somewhat more detailed proof which will be exploited again for
Theorem~\ref{thm:Qei} below. Fix a $\gamma$ as above such that
$B_T=B_T^\gamma$, and let $(A_0,A_1)$ be a disjoint pair of c.e. sets which
are e.i.\,. Consider (see for instance
\cite[Corollary~III.6.20]{Smorynski-logical}) two strict $\Sigma_{1}$
formulas $H_{0}(v), H_{1}(v)$ (i.e.\ each $H_{i}(v)$ is of the form $(\exists
w)\psi_{i}(w,v)$ with $\psi_{i}$ containing no quantifiers or only
quantifiers which are bounded by variables) such that $A_i= \{m: \mathbb{N}
\models H_{i}(\overline{m})\}$ (where $\overline{m}$ is the numeral term
corresponding to the number $m$) and let $H(v)$ be the formula $(\exists
w)(\psi_{0}(w,v) \wedge (\forall w'\leq w)\neg \psi_{1\check{}}(w',v))$.
Using \cite[Lemma~III.6.27]{Smorynski-logical} one can show that if $m \in
A_0$ then $\vdash_U H(\overline{m})$ and if $m \in A_1$ then $\vdash_U \neg
H(\overline{m})$, where $U$ is either $R$ or $Q$. Let now $(X_0, X_1)$ be the
disjoint pair of c.e. sets $X_0=\{m:\; \vdash_T \neg H(\overline{m})\}$ and
$X_1= \{m:\; \vdash_T H(\overline{m})\}$: since $(X_0, X_1) \supseteq (A_1,
A_0)$ we have that $(X_0,X_1)$ is e.i.: on the other hand the computable
function $m \mapsto \gamma(H(\overline{m}))$ gives a $1$-reduction of
$(X_0,X_1)$ to the pair $([0]_{T}, [1]_{T})$ (where $[x]_T$ denotes the
equivalence class of a number $x$ under the equivalence relation induced by
$\leq_T$): therefore the pair $([0]_T, [1]_T)$ is e.i. too, and thus $B_T$ is
an e.i. Boolean pre-algebra.
\end{proof}

Moreover, the associated quotient structure of $B_{T}$ is isomorphic to the
Lindenbaum sentence algebra of $T$, which leads us to the following
definition.

\begin{defn}\label{defn:prealgebraT}
$B_{T}$ as above is called the \emph{Boolean pre-algebra of the Lindenbaum
sentence algebra of $T$}.
\end{defn}
\end{eg}

For introductions to e.i. Boolean pre-algebras and related topics, see
(listed in chronological order) \cite{Nerode-Remmel:Survey,
Montagna-Sorbi:Universal, Nies:Effectively, Shavrukov}.

\subsection{C.e. precomplete pre-lattices and c.e. pre-lattices of sentences}
If $L$ is a c.e. precomplete bounded pre-lattice, then $\equiv_{L}$ is a
precomplete ceer and thus by Fact~\ref{fact:implications} yields a partition
of $\omega$ such that all pairs of distinct equivalence classes are e.i.\,.
Therefore $L$ is an e.i. pre-lattice.

\subsection{C.e. pre-lattices of sentences}
Let us consider a decidable first order language $L$ and a c.e. consistent
theory $T$ on $L$, and let $\mathcal{C} \subseteq \textrm{Sent}(L)$ be an
infinite c.e. set closed under the connectives $\wedge$ and $\lor$, and
intersecting both the set of theorems and the set of anti-theorems of $T$.
Fix a G\"odel numbering  $\gamma$  of $\mathcal{C}$ onto $\omega$ for which
$0$ and $1$ are the G\"odel numbers of a theorem and an anti-theorem,
respectively.

\begin{defn}\label{defn:pre-lattice-of-sentence}
The \emph{pre-lattice of $\mathcal{C}$-sentences relatively to $\gamma$} is
the c.e. bounded pre-lattice
\[
L_{\mathcal{C}/T}^\gamma=\langle \omega, \wedge^\gamma, \vee^\gamma, 0, 1,
\leq_{\mathcal{C}/T}^\gamma \rangle
\]
where $x \leq_{\mathcal{C}/T}^\gamma y$ if and only if $\vdash_T
\gamma^{-1}(x) \rightarrow \gamma^{-1}(y)$, and the operations are given (via
$\gamma$) by the propositional connectives. A \emph{pre-lattice of sentences}
is a pre-lattice of the form $L_{\mathcal{C}/T}^\gamma$, for some choice of
$T$, $\mathcal{C}$, and $\gamma$.
\end{defn}

Once again it is immediate to see that up to computable isomorphisms induced
by computable permutations of $\omega$ the definition does not depend on the
choice of $\gamma$, and thus, unless when strictly necessary for
definiteness, reference to the chosen $\gamma$ will always be omitted in the
following.

Also, up to a computable isomorphism provided by a computable permutation of
$\omega$ it does not matter whether we consider $\mathcal{C}$ or its closure
under provable equivalence in $T$: indeed, suppose that $\mathcal{C}$ is a
set of sentences providing the pre-lattice of sentences $L_{\mathcal{C}/T}$
(relatively, say, to $\gamma$), let
\[
[\mathcal{C}]_{T}=\{\sigma\in \textrm{Sent}(L): (\exists \gamma \in
\mathcal{C})[\vdash_{T} \gamma \leftrightarrow \sigma]\},
\]
and let us consider the pre-lattice of sentences $L_{[\mathcal{C}]_T/T}$
(relatively say to $\delta$). It is easy to see that $h= \delta \circ
\gamma^{-1} $ is a computable function satisfying for every $x,y$,
\[
x \leq_{\mathcal{C}/T} y \Leftrightarrow  h(x) \leq_{[\mathcal{C}]_{T}/T} h(y)
\]
and for every $y$ there exists $x$ such that $h(x)
\equiv_{[\mathcal{C}]_{T}/T} y$. By a standard back-and-forth argument, and
using that the equivalence classes of both $\equiv_{\mathcal{C}/T}$ and
$\equiv_{[\mathcal{C}]_{T}/T}$ are infinite, it is easy to see that from $h$
one can construct a computable permutation of $\omega$ inducing an
isomorphism of $L_{\mathcal{C}/T}$ with $L_{[\mathcal{C}]_T/T}$: At stage $0$
we define $f_0=\emptyset$. At stage $2s+1$, if $s \in \textrm{dom}(f_{2s})$
then let $f_{2s+1}=f_{2s}$; otherwise effectively search for the first $y$
for which $h(s) \equiv_{\mathcal{[C]_{T}}/T} y$ and $y \notin
\textrm{range}(f_{2s})$, and define $f_{2s+1}= f_{2s} \cup \{(s,y)\}$. At
stage $2s+2$, if $s \in \textrm{range}(f_{2s+1})$ then let $f_{2s+2}=
f_{2s+1}$; otherwise effectively search for the first pair $\langle u,
x\rangle$ for which $h(u) \equiv_{\mathcal{[C]_{T}}/T} s$, $x
\equiv_{\mathcal{C}/T} u$, and $x \notin \textrm{dom}(f_{2s+1})$, and define
$f_{2s+1}= f_{2s} \cup \{(x,s )\}$.

Pre-lattices of sentences for c.e. arithmetical first order theories $T$ have
been extensively studied by Shavrukov and Visser (see for instance
\cite{Shavrukov} and \cite{Shavrukov-Visser}), in particular taking
$\mathcal{C}$ to be the set of $\Sigma_n$-sentences for some $n\ge 1$. The
reader is referred to \cite{Smorynski-logical} for a clear introduction to
the classes $\Sigma_n$: the definition of $\Sigma_n$ is purely syntactic (and
in particular each class is decidable), depending only on the language and
not on any specific theory $T$, but our previous observation shows that
whether we take $\Sigma_n$ or its closure under provable equivalence in a
chosen $T$ we get pre-lattices that are isomorphic via an isomorphism
provided by a computable permutation of $\omega$ as long as in either
pre-lattice the pre-ordering relation is provided by $\vdash_T$.

It is easy to see that the quotient structure associated to $L_{\Sigma_n/T}$
is isomorphic to the Lindenbaum lattice of $\Sigma_{n}$-sentences of $T$,
i.e. the bounded sublattice of the Lindenbaum sentence algebra of $T$,
restricted to the equivalence classes of $\Sigma_{n}$-sentences. This
motivates the following definition.

\begin{defn}
$L_{\Sigma_n/T}$ will be called the \emph{pre-lattice of the Lindenbaum
lattice of $\Sigma_{n}$-sentences of $T$}.
\end{defn}

For the following theorem (where $EA$, called Elementary Arithmetic, is the
theory $I\Delta_0+ \textrm{Exp}$) see also
\cite[Example~4.1]{Shavrukov-Visser}.

\begin{thm}[\cite{Visser:Numerations}]\label{thm:EA-precomplete}
If $T$ is a c.e. consistent extension of elementary arithmetic $EA$  then,
for $n \ge 1$, $L_{\Sigma_n/T}$ is precomplete.
\end{thm}

\begin{proof}
If $\ulcorner \mbox{ }\urcorner$ is a conventional G\"odel numbering encoding
the syntactic objects and formulas of $L$ as normally used in the
arithmetization of metamathematics (representability of partial computable
functions, c.e. sets, etc.: see e.g.~ \cite{Smorynski-logical} or
\cite{Hajek-Pudlak} for details), and $\phi$ is a partial computable function
then from \cite{Visser:Numerations} we can see that there exists a computable
function $f_\phi$ which always land at $\ulcorner \mbox{ }\urcorner$-numbers
of $\Sigma_n$ sentences, and if $\phi(x)\downarrow = \ulcorner
\sigma\urcorner$, with $\sigma \in \Sigma_n$, then $f_\phi(x)=\ulcorner
\tau\urcorner$ with $\vdash_T \sigma \leftrightarrow \tau$. Let $\gamma$ be a
G\"odel numbering of the $\Sigma_n$-sentences onto $\omega$ so that
$L_{\Sigma_n/T}=L_{\Sigma_n/T}^\gamma$.  Now, given a partial computable
function $\phi$, consider the partial computable function $\psi$ with the
same domain as $\phi$ and $\psi(x)=\ulcorner \sigma \urcorner$ if
$\sigma=\gamma^{-1}(\phi(x))$, and let $g=\gamma(\tau)$ if
$f_\psi(x)=\ulcorner \tau \urcorner$: it is immediate to see that $g$ makes
$\phi$ total modulo $\equiv_T^\gamma$, showing that $\equiv_T^\gamma$ is
precomplete.
\end{proof}

As every c.e. bounded precomplete lattice is e.i., it follows that if $T$ is
a c.e. consistent extension of elementary arithmetic $EA$ then
$L_{\Sigma_n/T}$ is an e.i. pre-lattice. For Robinson's weaker systems $Q$
and $R$, it is not known if $L_{\Sigma_n/T}$ is precomplete. The following
theorem shows however that we still have an e.i. pre-lattice.

\begin{thm}\label{thm:Qei}
If $T$ is a c.e. consistent extension of $Q$ or $R$ then $L_{\Sigma_n/T}$ is
an e.i. pre-lattice.
\end{thm}

\begin{proof}
Let $n \ge 1$ be given, and fix a suitable G\"odel numbering $\gamma$ of the
$\Sigma_n$-sentences onto $\omega$  so that $L_{\Sigma_n/T}=
L_{\Sigma_n/T}^\gamma$. By consistency of $T$, $L_{\Sigma_n/T}$ is not
trivial. Notice that the formula $H(v)$ used in the proof of
Theorem~\ref{thm:BT} is $\Sigma_1$: therefore if $(X_0,X_1)$ is the e.i. pair
of c.e. sets as in that proof then the computable function $m \mapsto
\gamma(H(\overline{m}))$ gives a $1$-reduction of $(X_0,X_1)$ to the pair
$([0]_{\Sigma_n/T}, [1]_{\Sigma_n/T})$ which is therefore e.i.\,.
\end{proof}

\subsection{E.i. Heyting pre-algebras}\label{special-HA}
Examples of e.i. Heyting pre-algebras which are not Boolean pre-algebras come
from intuitionistic logic. For instance, Let $iT$ be any c.e. consistent
intuitionistic extension of Heyting Arithmetic $HA$ in the usual arithmetical
first order language $L$, and, having fixed a  G\"odel numbering $\gamma$
from $\textrm{Sent}(L)$ onto $\omega$ with $\gamma(\neg 0=0)=0$ and
$\gamma(0=0)=1$, let $\IT^\gamma=\langle \omega, \wedge^\gamma, \lor^\gamma,
\rightarrow^\gamma, \neg^\gamma, 0, 1, \leq_{iT}^\gamma\rangle$ be the c.e.
Heyting pre-algebra where the operations are the connectives (via G\"odel
numbers given by $\gamma$) and $x \le_{iT}^\gamma y$ if and only if
$\vdash_{iT} \gamma^{-1}(x) \rightarrow \gamma^{-1}(y)$. Once again, up to
computable isomorphisms induced by computable permutations of $\omega$, the
definition is independent of $\gamma$, so the superscript $\gamma$ will be
omitted unless when strictly necessary for a clearer understanding of the
text. As the Lindenbaum sentence algebra of $iT$ is isomorphic with the
quotient structure of $\IT$ it is fair to give the following definition.

\begin{defn}\label{defn:preHalgebras}
The Heyting pre-algebra $\IT$ is called the \emph{Heyting pre-algebra of the
Lindenbaum sentence algebra of $iT$}.
\end{defn}

By an argument similar to the proof of Theorem~\ref{thm:BT}, using similar
arithmetization tools (available in Heyting Arithmetic), one can show
effective inseparability of $\IT$. We choose however a different argument,
based on the G\"odel-Gentzen double-negation translation, which we
particularly like for its simplicity and clarity.

\begin{lemma}\label{lem:ex-Heyting}
$\IT$ is an e.i. Heyting pre-algebra.
\end{lemma}

\begin{proof}
Let $PA$ denote Peano Arithmetic, and let $B_{PA}$ be the Boolean pre-algebra
of the Lindenbaum sentence algebra of $PA$. The G\"odel-Gentzen
double-negation translation (see e.g. \cite{Avigad-Feferman} for mathematical
and historical details) ensures that there is a computable mapping $G$
transforming each arithmetical sentence $\sigma$ into an arithmetical
sentence $G(\sigma)$ such that $\vdash_{PA} \sigma$ if and only if
$\vdash_{HA} G(\sigma)$.  Moreover $G(\neg \sigma)=\neg G(\sigma)$, thus
$\vdash_{PA} \neg \sigma$ if and only if $\vdash_{HA} \neg G(\sigma)$ as
well. From this one easily gives a $1$-reduction $([0]_{PA}, [1]_{PA})\leq_1
([0]_{HA}, [1]_{HA})$ of disjoint pairs of sets, implying that $([0]_{HA},
[1]_{HA})$ is e.i.\, as so is $([0]_{PA}, [1]_{PA})$: therefore the pair
$([0]_{iT}, [1]_{iT})$ is e.i. too as it contains the pair $([0]_{HA},
[1]_{HA})$.
\end{proof}

\begin{rem}
There is of course a much wider range of examples of e.i. Heyting
pre-algebras than the ones covered by Lemma~\ref{lem:ex-Heyting}: forthcoming
work will address the issue of examples of e.i. Heyting pre-algebras arising
from intuitionistic versions of weak fragments of arithmetic.
\end{rem}

\subsection{E.i. free distributive pre-lattices}\label{sssct:eidistr}
Useful examples (and also counterexamples) of e.i. pre-structures come from
free pre-structures.

\begin{defn}
Define a pre-structure to be \emph{free on a countably infinite set of
generators} if the corresponding associated quotient structure is free on a
countably infinite set of generators.
\end{defn}

Consider a computable presentation of the form $F^d_{0,1}(X)=\langle \omega,
\wedge, \lor, 0,1\rangle$ of the free bounded distributive lattice on a
decidable infinite set $X$ (with $0,1 \notin X$)) which is computably listed
without repetitions by $\{x_i: i \in \omega\}$. Fix a pair $(U,V)$ of e.i.
c.e. sets, and let $\alpha$ be the c.e. congruence on $F^d_{0,1}(X)$
generated by the set of pairs $\{(x_i,0): i \in U\} \cup \{(x_i, 1): i \in
V\}$. (This argument to encode effective inseparability in a c.e. quotient of
$F^d_{0,1}(X)$ has been suggested in private communication by Shavrukov.)
Define $L^d_{01}= \langle \omega, \wedge, \lor, 0,1, \leq_{L^d_{01}} \rangle$
where $\wedge, \lor$ are the same operations as in $F^d_{0,1}(X)$, and $x
\leq_{L^d_{01}} y$ if $[x]_{\alpha}= [x \wedge y]_{\alpha}$. Throughout the
remainder of this section leading to Example~\ref{ex:eidistr}, for simplicity
denote $L=L^d_{01}$. It is easy to see that $L$ is a c.e. pre-lattice whose
associated quotient structure is $F^d_{0,1}(X)_{/\alpha}$. On the other hand
we can argue that $F^d_{0,1}(X)_{/\alpha}$ is isomorphic with
$F^d_{0,1}(\{x_i: i \notin U\cup V\})$: if $f:X \longrightarrow
F^d_{0,1}(\{x_i: i \notin U\cup V\})$ is the function mapping $x_i \mapsto 0$
if $i \in U$, $x_i \mapsto 1$ if $i \in V$, and $x_i \mapsto x_i$ if $i
\notin U \cup V$, then by freeness $f$ extends to a (unique) homomorphism $g:
F^d_{0,1}(X) \longrightarrow F^d_{0,1}(\{x_i: i \notin U\cup V\})$ which is
onto as $f$ is onto a generating set, and it is not difficult to see that
$\ker(g)=\alpha$, i.e. $g(x)=g(y)$ if and only if $x \rel{\alpha} y$. Thus by
the First Isomorphism Theorem of universal algebra we see that
$F^d_{0,1}(X)_{/\alpha} \simeq F^d_{0,1}(\{x_i: i \notin U\cup V\})$. In
conclusion, $L$ is free bounded distributive on a countably infinite set of
generators; moreover the function $f(i)=x_i$ provides a $1$-reduction of the
e.i. pair $(U,V)$ to the pair $([0]_L, [1]_L)$, giving that $L$ is e.i.\,.

\medskip

We have shown:
\begin{eg}\label{ex:eidistr}
$L^{d}_{01}$ is an e.i. distributive pre-lattice, which is a free bounded
distributive pre-lattice on countably many generators.
\end{eg}

\subsection{E.i. non-distributive pre-lattices}\label{sssct:einondistr}
The same construction as in Section~\ref{sssct:eidistr} but starting with the
free bounded non-distributive lattice $F^{01}_{nd}(X)$ on $X$, leads to a
c.e. pre-lattice $L_{01}^{nd}$.

\begin{eg}\label{ex:nondistr}
$L_{01}^{nd}$ is an e.i. non-distributive pre-lattice which is free bounded
on countably many generators.
\end{eg}

\section{Effectively inseparable pre-lattices and universality}
Montagna and Sorbi~\cite[Theorem~2.1]{Montagna-Sorbi:Universal} proved that
if $L$ is an e.i. pre-lattice and $R=\langle \omega, \leq_{R}\rangle$ is a
decidable pre-partial order then $\leq_{R} \preceq \leq_{L}$, but left open
(\cite[Remark~2.1]{Montagna-Sorbi:Universal}) whether such a pre-partial
order $\leq_{L}$ is universal. In the next section we show that this is so.

\subsection{E.i. pre-lattices vs. uniform finite precompleteness}
We first show that effective inseparability of the pair $([0]_{L}, [1]_{L})$
makes the ceer $\equiv_L$ (and thus $\le_{L}$, according to
Definition~\ref{defn:ei-preorders-and-other-things}) u.f.p.\,.

\begin{thm}\label{thm:from-ei-to-ufp}
Let $L$ be a c.e. pre-lattice. Then $L$ is e.i. if and only if the ceer
$\equiv_L$ is u.f.p.\,.
\end{thm}

\begin{proof}
If $\equiv_L$ is u.f.p. then by Fact~\ref{fact:implications} all pairs of
distinct equivalence classes are e.i.\,.

Suppose now that $L$ is e.i.\,. We need to define a computable function
$k(D,e,x)$ such that if $\phi_e(x)\downarrow \equiv_L d$ for some $d \in D$
then $k(D,e,x) \equiv_L \phi_e(x)$.

Let $p$ be a productive function for the pair $([0]_L, [1]_L)$. Let
\[
\{u_{d,D,e,x}, v_{d,D,e,x}: \text{ $D$ finite subset of $\omega$, }
d\in D, e,x \in \omega\}
\]
be a computable set of indices we control by the Recursion Theorem. For a
pair $(u_{d,D,e,x}, v_{d,D,e,x})$ in this set let $
c_{d,D,e,x}=p(u_{d,D,e,x},v_{d,D,e,x})$ and $a_{d,D,e,x}=d \wedge
c_{d,D,e,x}$. Define
\[
k(D,e,x)=\bigvee_{d \in D}a_{d,D,e,x}.
\]
We now specify how to computably enumerate $W_{u_{d,D,e,x}}$ and
$W_{v_{d,D,e,x}}$ for $d\in D$. We wait for $\phi_e(x)$ to converge to some
$y$ which is $\equiv_{L}$ to some element in $D$, and while waiting, we let
$W_{u_{d,D,e,x}}$ and $W_{v_{d,D,e,x}}$ enumerate $[0]_L$ and $[1]_L$,
respectively. If we wait forever then for all $d \in D$ we end up with
$W_{u_{d,D,e,x}}=[0]_L$ and $W_{v_{d,D,e,x}}=[1]_L$. If the wait terminates,
let $d_0\in D$ be the first seen so that $\phi_e(x)\equiv_{L} d_0$, and then
enumerate $c_{d_0,D,e,x}$ into $W_{u_{d_0,D,e,x}}$, while keeping on with
$W_{u_{d_0,D,e,x}}$ enumerating $[0]_L$ and with $W_{v_{d_0,D,e,x}}$
enumerating $[1]_L$: this ends up with  $W_{u_{d_0,D,e,x}}=[0]_L \cup
\{c_{d_0,D,e,x}\}$ and $W_{v_{d_0,D,e,x}}=[1]_L$, thus forcing $c_{d_0,D,e,x}
\in [1]_L$ and thus $a_{d_0,D,e,x}\equiv_{L} d_0$. For all $d \in D$ with $d
\ne d$, we let $W_{u_{d,D,e,x}}=[0]_L$ and $W_{v_{d,D,e,x}}=[1]_L \cup
\{c_{d,D,e,x}\}$: this forces $c_{d,D,e,x}\in [0]_L$ and thus
$a_{d,D,e,x}\equiv_{L} 0$ for each such $d$.

In conclusion, if $\phi_e(x)$ converges to a number that is
$\equiv_{L}$-equivalent to some $d \in D$ then, for the number $d_0$ for
which we see this for the first time, we have
\[
k(D,e,x)=\bigvee_{d \in D}a_{d,D,e,x} \equiv_L d_0,
\]
as desired.
\end{proof}

\subsection{E.i. pre-lattices are universal}
We now show that effective inseparability for a c.e. pre-lattice implies
universality with respect to the class of all c.e. pre-orders.

\begin{thm}\label{thm:ei-lattices-universal}
If $L=\langle \omega, \wedge, \vee, 0,1, \leq_{L}\rangle$ is an e.i.
pre-lattice then $\leq_{L}$ is universal.
\end{thm}

\begin{proof}
By Theorem~\ref{thm:from-ei-to-ufp} let $j(D,e,x)$ be a function witnessing
that $\equiv_{L}$ is u.f.p.\,.

\begin{claim}\label{claim:ufp}
There is a computable function $k(a,b,D,e,x)$ so that if $a\leq_{L} b$, then
$k(a,b,D,e,x)\in [a,b]_{\leq_L}$, and if $\phi_e(x)\equiv_L  d\in D$ for some
$d\in D$ then $k(a,b,D,e,x)\equiv_L (d\wedge b) \vee a$. In particular, if
$\phi_e(x)\equiv_L  d\in D$ for some $d\in D$ with $d\in [a,b]_{\leq_L}$,
then $k(a,b,D,e,x)\equiv_L  d$.
\end{claim}

\begin{proof}
Let $k(a,b,D,e,x)=(j(D,e,x)\wedge b)\vee a$. Then if $a\leq_{L} b$, then
$k(a,b,D,e,x)\in [a,b]_{\leq_L}$. If $\phi_e(x)\equiv_L  d\in D$, then
$k(a,b,D,e,x)\equiv_L (d\wedge b)\vee a$. Thus, if $d\in [a,b]_{\leq_L}$,
then $k(a,b,D,e,x)\equiv_L  d$.
\end{proof}

Since we know that $\equiv_L$ has this behavior with regard to $k$, we use a
``speedup'' of the enumeration of $\equiv_L$ so that this effect happens
immediately. That is, whenever at a stage $s$ we need to define, given
$a,b,D$, some computation $\phi_e(x)$ to converge to some $d\in D$ relying on
the fact that this entails $k(a,b,D,e,x)\equiv_L (d\wedge b)\vee a$, we
momentarily stop the construction until we in fact see $\equiv_L$ to give
this equivalence, and only after that we proceed with the construction: we
may therefore pretend that when we make $\phi_e(x)$ converge then at the same
stage we also see $k(a,b,D,e,x)\equiv_L (d\wedge b)\vee a$.

Let $R=\langle \omega, \leq_{R}\rangle$ be any c.e. pre-order on $\omega$. We
must construct a computable function $f:\omega \rightarrow \omega$ so that
$n\leq_R m$ if and only if $f(n)\leq_L f(m)$. We fix an infinite set $E$ of
indices that we control via the Recursion Theorem. The idea is then to define
the map $f$ reducing $\leq_R$ to $\leq_L$ inductively. Having defined $f(j)$
for each $j<n$, we define $f(n)$ as follows. We consider first the partial
order of the set $T^0_n$ of all possible pre-partial orders on the set
$\{j\mid j<n\}$ ordered by $O_1 \preceq^0_n O_2$ if $O_2$ extends $O_1$, that
is, if $i\leq_{O_1} j$ implies $i \leq_{O_2}j$ for each $i,j<n$. This defines
a finite partial order $(T_n^0,\preceq^0_n)$. Next we consider the partial
order $(T_n, \preceq_n)$ where
\begin{multline*}
T_n=\{(O,X,Y)\mid O\in T_n^0, X,Y\subseteq \{j\mid j< n\}\\ \text{ and }
a\leq_{O} b \text{ for every $a\in X$ and $b\in Y$} \\
\text{and if $i\leq_{O} a\in X$,
then $i\in X$, and if $j\geq_{O} b\in Y$ then $j\in Y$}\}.
\end{multline*}
In other words, $T_n$ is comprised of pre-partial orders of the set $\{j\mid
j<n\}$ with a chosen ``interval''. We say $(O_1,X_1,Y_1)\preceq_{n}
(O_2,X_2,Y_2)$ if $O_1\preceq^0_n O_2$ and $X_1\subseteq X_2$ and
$Y_1\subseteq Y_2$. In other words, we have refined both the order and the
interval. Once again, $(T_n,\preceq_n)$ is a finite partial order.

For each element $\tau=(O,X,Y)\in T_n$, we define an $x_{\tau}\in \omega$. We
begin by defining $x_\tau$ for the leaf $\upsilon$ in $T_n$. Note that the
leaf corresponds to a pre-partial order where all elements are
$O$-equivalent, $X=Y=\{j\mid j<n\}$. We define $x_\upsilon$ to be
\[
x_\upsilon= k(\bigvee_{i\in X}f(i),\bigwedge_{j\in Y}f(j),\{f(j)\mid
j<n\}
\cup \{0,1\},e_\upsilon,0)
\]
where $e_\upsilon\in E$ is distinct from any previously mentioned element of
$E$.

For every node $\tau\in T_{n}$, we define
\[
x_\tau=k(\bigvee_{i\in X}f(i),\bigwedge_{j\in Y}f(j), \{x_\sigma\mid
\tau\preceq_n \sigma\}
\cup \{0,1\}, e_\tau,0)
\]
where $e_\tau\in E$ is distinct from any previously mentioned element of $E$.

We fix notation that $\tau=(O_\tau,X_\tau,Y_\tau)$ and
$x_\tau=k(a_\tau,b_\tau,S_\tau,e_\tau,0)$. Let $\lambda$ be the root in
$T_n$, i.e. the partial order given by an anti-chain and $X=Y=\emptyset$.

We define $f(n)$ to be the element $x_\lambda$, where $\lambda$ is the root
in $T_n$ ($\lambda=\lambda(n)$ depends in fact on $n$). The idea is to use in
a construction in stages the indices $e_\tau$ for various $\tau$ to ensure
that: If the elements $\{j\mid j<n\}$ in $\leq_R$ have order type $O$, and
$X=\{j\mid j<n\wedge j\leq_R n\}$ and $Y=\{j\mid j<n\wedge j\geq_R n\}$ at
stage $s$, then $x_\lambda \equiv_{L} x_{(O,X,Y)}$. Since there are only
finitely many elements of $T_n$, we will only need to act at finitely many
stages to ensure this is true ($(O,X,Y)=\tau(n,s)$ depends in fact on $n$ and
$s$).

\medskip

We work with computable approximations $\{\leq_{R,s}: s \in \omega\}$,
$\{\leq_{L,s}: s \in \omega\}$ to $\leq_R$, $\leq_L$ respectively (and
consequent approximations $\{\equiv_{R,s}: s \in \omega\}$, $\{\equiv_{L,s}:
s \in \omega\}$ to $\equiv_R$, $\equiv_L$ respectively): we may assume that
these sequences are increasing sequences of pre-orders $\leq_{R,s}$,
$\leq_{L,s}$ with unions $R,S$ respectively, such that the predicates, in
$x,y,s$, $x \leq_{R,s} y$, $x \leq_{L,s} y$ are decidable and $\leq_{R,0},
\leq_{L,0}$ are the identity pre-orders on $\omega$. In the construction
below it is understood that at each stage $s$, when we monitor the c.e.
progress of the two c.e pre-orderings and of the corresponding equivalence
relations, we work with the corresponding approximations at that stage.

Recall that $f(n)=x_{\lambda(n)}$ where $\lambda(n)$ is the root in $T_n$.

\medskip

\emph{Stage $0$.} Do nothing.

\medskip

\emph{Stage $s+1$.}  Let $n$ be given, and suppose we have dealt already with
all $j<n$. We now deal with $n$. Recall that $f(n)=x_\lambda$. Let $O_1$
represent the order type of $R$ restricted to $\{j\mid j<n\}$ at stage $s$
and $O_2$ be the order type of $R$ restricted to $\{j\mid j<n\}$ at stage
$s+1$. Let $X_1$ be the set of elements in $\{j\mid j<n\}$ which are $\leq_R
n$ at stage $s$ and let $Y_1$ be the set of elements in $\{j\mid j<n\}$ which
are $\geq_R n$ at stage $s$. Let $X_2,Y_2$ be defined similarly at stage
$s+1$. Then $(O_1,X_1,Y_1)\preceq_n (O_2,X_2,Y_2)$.

If $(O_1,X_1,Y_1)\ne(O_2,X_2,Y_2)$ then make $\phi_{e_{(O_1,X_1,Y_1)}}(0)$
converge to equal $x_{(O_2,X_2,Y_2)}$. Note that this causes (by our speedup)
that $x_{(O_1,X_1,Y_1)}\equiv_L x_{(O_2,X_2,Y_2)}$ also at stage $s+1$. (We
will show by induction that then $f(n)=x_\lambda$ will become
$\equiv_{L}$-equivalent to $x_{(O_2,X_2,Y_2)}$, which will ensure that $f(n)
\leq_L f(j)$ for every $j\in Y_2$ and $f(n) \geq_L f(j)$ for every $j\in
X_2$.)

If $(O_1,X_1,Y_1)=(O_2,X_2,Y_2)$ then go to $n+1$.

\medskip

If ever we see, at any stage, $f(n)\leq_L f(m)$ when we have not yet seen
$n\leq_R m$, then we halt the entire construction and call action
$\diamondsuit$. Action $\diamondsuit$, if ever called, will cause a
contradiction, which will guarantee that this will never happen and action
$\diamondsuit$ will never in fact be called.

\subsection*{Action $\diamondsuit$}
We now describe action $\diamondsuit$. Suppose that at stage $s$ of the
construction, we see $f(n)\leq_L f(m)$ and we have not yet seen $n\leq_R m$.
In particular, at this stage $f(n)\equiv_L x_\tau$ and $f(m)\equiv_L
x_\sigma$ where $\tau=\tau(n,s)$ and $\sigma=\tau(m,s)$ are as given for $n$
and $m$ by Claim~\ref{claim:fund}, so that we have not as yet caused
$\phi_{e_\tau}(0)$ or $\phi_{e_{\sigma}}(0)$ to converge. Recall that
$a_\tau=\bigvee_{i\in X_\tau}f(i)$ and $b_\tau=\bigwedge_{j\in Y_\tau}f(j)$,
and similarly for $\sigma$.

Now, for every $z\in \{n\}\cup Y_\tau$, we perform the following action,
which we call $\text{One}(z)$: take $\rho=\tau(z,s)$ as given for $z$ again
by Claim~\ref{claim:fund} (in particular $x_\rho\equiv_L f(z)$ and
$\phi_{e_\rho}(0)\uparrow$ at this stage). Now, make $\phi_{e_{\rho}}(0)$
converge to equal $1$. This forces $f(z) \equiv_L b_\rho\vee a_\rho$ (by our
speedup this happens at $s$). In other words, we cause at $s$ $f(z)$ to
become as large as possible.

Also, for every $z\in \{m\}\cup X_\sigma$, we perform the following action,
which we call $\text{Zero}(z)$: take $\rho$ as given for $z$ again by
Claim~\ref{claim:fund} (in particular $x_\rho\equiv_L f(z)$ and
$\phi_{e_\rho}(0)\uparrow$ at this stage). Now, make $\phi_{e_{\rho}}(0)$
converge to equal $0$, which causes $f(z) \equiv_L a_\rho$ (by our speedup
this happens at $s$). In other words, we cause at $s$ $f(z)$ to become as
small as possible.

\subsection*{Verification}
We now verify that the above construction yields a function
$f:\omega\rightarrow \omega$ so that $n\leq_R m$ if and only if $f(n)\leq_L
f(m)$.

\begin{claim}\label{claim:fund}
Suppose that we have not called action $\diamondsuit$ at any stage $<s$, and
let $m$ be given. Let $\tau(m,s)=(O,X,Y) \in T_m$ be the triple defined by:
$O$ is the order type of $\{j\mid j<m\}$ in $R$ at stage $s$, $X=\{j\mid j<m
\text{ and }j\leq_R m\}$ at stage $s$, and $Y=\{j\mid j<m \text{ and }
j\geq_R m\}$ at stage $s$. Then (letting for simplicity $\tau=\tau(m,s)$)
$f(m) \equiv_L x_\tau$ at stage $s$ and $\phi_{e_{\tau}}(0)\uparrow$ at stage
$s$, and $f$ preserves the pre-partial order given by $R$ at stage $s$ on the
interval $\{j: j \leq m\}$.
\end{claim}

\begin{proof}
At stage $0$, we have $f(m)=x_\lambda$, so the claim holds for every $m$ at
stage $0$.

Assume by induction that the claim holds for every $m$ at stage $s$ and for
every $j <m$ at stage $s+1$.

Notice that we can act at stage $s+1$, as the construction has not been
stopped due to $\diamondsuit$. Let $\tau(m,s)=(O,X,Y)$ be the triple defined
as in the statement of the Claim, i.e. $O$ is the order type of $\{j\mid
j<m\}$ in $R$ at stage $s$, $X=\{j\mid j<m \text{ and }j\leq_R m\}$ at stage
$s$ and $Y=\{j\mid j<m \text{ and } j\geq_R m\}$ at stage $s$. By induction,
if $\tau=\tau(m,s)$ then at $s$ $f(m) \equiv_{L} x_{\tau}$. Let
$\sigma=\tau(m,s+1)$ be the triple defined analogously at stage $s+1$, so we
make $\phi_{e_{\tau}}(0)\downarrow =x_\sigma$. But then, by our speedup, at
$s+1$
\[
x_\tau=k(a_\tau,b_\tau,S_\tau,e_\tau,0)\equiv_L (x_\sigma\wedge b_\tau)\vee
a_\tau.
\]
Moreover, by inductive hypothesis at $s+1$ (on $f$ preserving at $s+1$ the
pre-order of $R$ on numbers smaller than $m$), $a_\tau \le_L b_\tau$ and
$a_\sigma \le_L b_\sigma$: this implies by property of function $k$ that
$a_\tau \le_L x_\tau \le_L b_\tau$ and $a_\sigma \le_L x_\sigma \le_L
b_\sigma$. Now, since $X_\tau\subseteq X_\sigma$ and $Y_\tau\subseteq
Y_\sigma$, we get (all the following $\equiv_L$-equivalences by our speedup
may be assumed to be seen at $s+1$)
\[
a_\tau \le_L a_\sigma \le_L b_\sigma \le_L b_{\tau},
\]
and so at stage $s+1$ we have that $f(m)\equiv_L  x_\tau\equiv_L  x_\sigma$:
indeed $x_{\sigma} \wedge b_{\tau}\equiv_{L}  x_{\sigma}$ and $x_{\sigma}
\lor a_{\tau}\equiv_{L} x_{\sigma}$. Further, there is no condition which
could have caused $\phi_{e_\sigma}$ to have converged before this stage, so
we still have $\phi_{e_\sigma}(0)\uparrow$ at stage $s+1$.

We finally show that $f$ preserves the pre-order given at stage $s+1$ by $R$
on $\{j: j \leq m\}$. By inductive hypothesis, the pre-order given at stage
$s+1$ by $R$ on $\{j: j < m\}$ is preserved. Also, at $s+1$, since
$f(m)\equiv_L x_{\sigma}$ we have that $a_\sigma\leq_L f(m) \leq_L b_\sigma$,
so all the inequalities which occur between $m$ and any $j<m$ in $R$ are
preserved in $\leq_L$ at $s+1$.
\end{proof}

It now remains to show that we never call action $\diamondsuit$. Then, Claim
\ref{claim:fund} shows that $f$ preserves the pre-order from $R$ and the fact
that we do not call action $\diamondsuit$ shows that we never have any more
inequalities holding in $\leq_L$ than in $\leq_R$. Thus $f$ is a reduction of
$\leq_R$ to $\leq_L$.

\begin{claim}
We never call action $\diamondsuit$.
\end{claim}

\begin{proof}
Suppose we call action $\diamondsuit$ at stage $s$ because we see $f(n)\leq_L
f(m)$ and we have not yet seen that $n\leq_R m$. By calling action
$\text{One}(z)$ for every $z\in \{n\}\cup Y_\tau$, we arrange that for each
of these elements $z$, $f(z)$ become (by our speedup this happens at $s$)
$\equiv_L$-equivalent to $b_\rho\vee a_\rho$, as in the construction (with
$\rho, \tau, \sigma$ given for $m,n,z$ at $s$ by Claim~\ref{claim:fund}),
$f(m)\equiv_{L} x_{\tau}$, $f(n)\equiv_{L} x_{\sigma}$, $f(z)\equiv_{L}
x_{\rho}$ and we have not as yet caused either of $\phi_{x_{\tau}}(0)$,
$\phi_{x_{\sigma}}(0)$, $\phi_{x_{\rho}}(0)$ to converge. But since at $s$
the order in $L$ is the image of the order in $R$ (since we have not as yet
called action $\diamondsuit$), we get that $f(z) \equiv_L b_\rho$, since
$b_\rho \geq_L a_\rho$. We now show that for every $z\in \{n\}\cup Y_\tau$,
$f(z)$ becomes $\equiv_L$-equivalent to $1$ (in the following $\equiv_L$ and
$\leq_R$ are understood to be approximated at stage $s$). Suppose towards a
contradiction that for some element in this set, the image under $f$ is not
$\equiv_L$-equivalent to $1$, and let $z$ be $\leq_{R}$-maximal among those
such that $f(z)$ is not $L$-equivalent to $1$. Subject to this restraint, we
further choose $z$ to be minimal (as a number in $\omega$). Now, if $y \in
Y_{\rho}$ (so that $y<z$ and thus $y<n$ as $z \leq n$ being $z\in Y_\tau \cup
\{n\}$) then $y \ge_{R} z \ge_{R} n$ (if $z=n$ then at $s$ $z \ge_R n$ by our
assumptions on the approximations to $\leq_R$) and thus $y \in Y_\tau$: by
$R$-maximality of $z$, we see that either $y\equiv_R z $ or $f(y)\equiv_L 1$.
In the former case, we have $f(y)\equiv_L f(z)$ (since at $s$ $f$ preserves
the pre-order) and that $z$ was not minimal in $\{n\}\cup Y_\tau$ with
$f(z)\not\equiv_L 1$. Thus, we always have the latter case, and we see that
for every $y \in Y_\rho$ we have $f(y)\equiv_L 1$. But $f(z) \equiv_L b_\rho$
which is the meet of the images of the members of $Y_\rho$, and thus
$\equiv_L 1$. Thus, $f(z) \equiv_L 1$.

The same proof shows that for every $z\in \{m\}\cup X_\sigma$, $f(z)$ becomes
$\equiv_L$-equivalent to $0$. Thus we have that $1 \equiv_L f(n) \leq_L f(m)
\equiv_L 0$ (taking in turn $z=n$ and $z=m$), which is a contradiction. Thus,
the action $\diamondsuit$ is never called.	
\end{proof}
This ends the proof.
\end{proof}

\begin{rem}
It should be noted that the proofs of Theorem~\ref{thm:from-ei-to-ufp} and
Theorem~\ref{thm:ei-lattices-universal} do not require distributivity in
their assumptions. On the other hand, for a c.e. pre-lattice distributivity
is not a consequence of being e.i.: Example~\ref{ex:nondistr} shows in fact
that there are e.i. non-distributive pre-lattices.

Notice however that the proofs of Theorem~\ref{thm:from-ei-to-ufp} and
Theorem~\ref{thm:ei-lattices-universal} make use of the existence of both
operations of $\wedge$ and $\lor$. The following observation shows that this
is indeed necessary. A bounded c.e. upper pre-semilattice $U=\langle \lor,
0,1, \leq_U \rangle$ is called \emph{effectively inseparable} (\emph{e.i})
if, again, the pair $([0]_U,[1]_U)$ is effectively inseparable.
\end{rem}

\begin{obs}\label{obs:not-for-semilattices}
There exists an e.i. upper pre-semilattice whose pre-ordering relation is not
universal.
\end{obs}

\begin{proof}
The same construction as in Section~\ref{sssct:eidistr} but starting with the
free bounded upper pre-semi-lattice on $X$ leads to a c.e. bounded upper
pre-semilattice $U^-_{01}$ which is e.i. and as a bounded upper
pre-semilattice is free on countably many generators.

We claim that the c.e. order $\langle \omega, \leq^*\rangle$ where $x \leq^*
y$ if $x\ge y$ can not be reduced to $\leq_{U^-}$, thus showing that
$\leq_{U^-_{01}}$ is not universal. This is so because there is no infinite
descending chain in $U^-_{01}$, as every element other than 1 is the join of
a unique finite set of generators, which are atoms in $U^-_{01}$.
\end{proof}

\subsection{Doing without boundedness}
We have shown (Theorem~\ref{thm:from-ei-to-ufp}) that every e.i. pre-lattice
is u.f.p. and thus u.e.i. by Fact~\ref{fact:implications}. In this section we
reverse the last implication by showing that for c.e. pre-lattices
u.e.i.-ness, even without boundedness, is enough to have u.f.p.-ness.

\begin{lemma}\label{lem:ei-restriction}
Let $L=\langle \omega, \wedge, \vee, 0,1, \leq_{L}\rangle$ be a c.e.
pre-lattice, with $a,b$ such that $a <_{L} b$ and the pair $([a]_{L},
[b]_{L})$ is e.i.\,. Then there exists a productive function $p$ for this
pair such that if $[a]_{L} \subseteq W_{u}$, $[b]_{L} \subseteq W_{v}$ and
$W_{u}\cap W_{v}=\emptyset$ then $p(u,v)\in [a,b]_{\leq_L}$.
\end{lemma}

\begin{proof}
Let $q$ be a productive function for the pair $([a]_{L}, [b]_{L})$. We use
two indices $u',v'$ (uniformly depending on $u,v$ respectively) which we
control by the Recursion Theorem. We program enumerations of $W_{u'}$ and
$W_{v'}$ as follows. Let $j=(a \lor q(u',v')) \wedge b$. Then $W_{u'}$ and
$W_{v'}$ enumerate $[a]_{L}$ and $[b]_{L}$ respectively: if we see $j\in
W_{u}$ then we also enumerate $q(u',v')$ into $W_{u'}$; if we see $j\in
W_{v}$ then we also enumerate $q(u',v')$ into $W_{v'}$. In other words,
\begin{align*}
W_{u'}&=
\begin{cases}
[a]_L, &\text{if $j \notin W_u$},\\
[a]_L \cup \{q(u'v')\}, &\text{otherwise},
\end{cases}
\\
W_{v'}&=
\begin{cases}
[b]_L, &\text{if $j \notin W_v$},\\
[b]_L \cup \{q(u'v')\}, &\text{otherwise}.
\end{cases}
\end{align*}

Assume now that $[a]_{L} \subseteq W_{u}$ and $[b]_{L} \subseteq W_{v}$ and
$W_{u}\cap W_{v}=\emptyset$. Then $j \notin W_{u} \cup W_{v}$: if for
instance $j \in W_{u}$ then we have forced $q(u',v') \in [b]_{L}$ giving
$j\in [b]_{L}\subseteq W_{v}$, thus $W_{u} \cap W_{v}\ne \emptyset$,
contradiction. Thus, $p(u,v)=j$ is our desired productive function.
\end{proof}

\begin{lemma}\label{lem:restriction}
Let $L=\langle \omega, \wedge, \vee, 0,1, \leq_{L}\rangle$ be a c.e.
pre-lattice, with $a,b$ such that $a <_{L} b$ and $([a]_{L}, [b]_{L})$ is
e.i.\,. Then the interval $[a,b]_{L}$ is computably isomorphic to an e.i.
pre-lattice.
\end{lemma}

\begin{proof}
We first observe that the interval $[a,b]_{L}$ is infinite, as the
equivalence classes of $a$ and $b$ are infinite. We argue as in
Remark~\ref{rem:more-general}. Let $h: \omega \longrightarrow [a,b]_{L} $ be
a computable bijection with $h(0)=a$ and $h(1)=b$. This induces a c.e.
pre-lattice $L'=\langle \omega, \wedge, \vee, 0,1, \leq_{L'}\rangle$ where
$x\le_{L'} y$ if and only if $h(x) \le_{L} h(y)$. By
Lemma~\ref{lem:ei-restriction} the equivalence classes $[0]_{L'}, [1]_{L'}$
form an e.i. pair: if $[0]_{L'}\subseteq W_{u}$ and $[1]_{L'}\subseteq
W_{v}$, and $W_{u} \cap W_{v}$ are disjoint, then $h[W_{u}]$ and $h[W_{v}]$
are disjoint supersets of $[a]_{L}, [b]_{L}$ respectively, so that by the
lemma (which provides a productive function which always lands in $[a,b]_{L}$
when applied to indices of disjoint c.e. supersets of $[a]_{L}, [b]_{L}$) we
can effectively find $y\in [a,b]_{L}$ so that $h^{-1}(y)$ is defined but $y
\notin h[W_{u}] \cup h[W_{v}]$, and thus $h^{-1}(y) \notin W_{u}\cup W_{v}$.
\end{proof}

Theorem~\ref{thm:from-uei-to} and Observation~\ref{obs:open} below answer
questions raised in private communication by Shavrukov.

\begin{thm}\label{thm:from-uei-to}
A c.e. pre-lattice is u.e.i. if and only if it is u.f.p.\,.
\end{thm}

\begin{proof}
The implication $\Leftarrow$ follows from Fact~\ref{fact:implications}. For
the other implication, let $L=\langle \omega, \wedge, \lor, 0,1, \leq_L
\rangle$ be a c.e. pre-lattice which is u.e.i.\,. By
Remark~\ref{rem:uniformity-for-uei} from any pair $a,b$ we can uniformly find
an index of a total computable function $q$ which is a productive function
for the pair $([a]_L,[b]_L)$ if the two equivalence classes are distinct. A
close look at Lemmata~\ref{lem:ei-restriction}~and~\ref{lem:restriction}
shows that their proofs are uniform: namely, (Lemma~\ref{lem:ei-restriction})
if $L$ is u.e.i then from any pair $a,b$ we can uniformly find a computable
function $q_{a,b}$ (use totality of $q$) which has the property therein
stated if $a<_L b$; and thus (Lemma~\ref{lem:restriction}) from any pair
$a,b$ we can uniformly go to an index of a partial computable function
$h_{a,b}$, partial computable binary functions $\wedge_{a,b}, \lor_{a,b}$, a
c.e. relation $\leq_{a,b}$, and a computable function $p_{a,b}$ such that if
$a \nleq_L b$ then $h_{a,b}$ never converges; if $a<_L b$ then
$L_{a,b}=\langle \omega, \wedge_{a,b}, \lor_{a,b}, 0,1, \leq_{a,b}\rangle$ is
an e.i. pre-lattice such that $h_{a,b}: [a,b]_L \longrightarrow \omega$ is a
bijection giving an order-theoretic isomorphism of the sublattice of $L$
having universe $[a,b]_L$ with $L_{a,b}$, and $p_{a,b}$ is a productive
function for $([0]_{L_{a,b}},[1]_{L_{a,b}})$; if $a \equiv_L b$, then
$h_{a,b}$ is a bijection between $[a]_L$ and $\omega$. It follows that if $a
\leq_L b$ then $h_{a,b}$ is a bijection between $[a,b]_L$ and $\omega$, and
thus from the canonical index of a finite $D \subseteq [a,b]_L$ we can
effectively find the canonical index of the image $h_{a,b}[D]$. Finally, by
the proof of Theorem~\ref{thm:from-ei-to-ufp} (which is uniform) from $a,b$,
using totality of $p_{a,b}$, we can uniformly find an index of a total
computable function $k_{a,b}(D,e,x)$ which witnesses that $L_{a,b}$ is
u.f.p.\, if $a<_L b$.

Given now any finite set $D$ let $a_D=\bigwedge D$ and $b_D=\bigvee D$ where
the meet and the join are taken in $L$: notice that $a_D \leq_L b_D$, thus
$h_{a_D,b_D}$ is a bijection between $[a_D, b_D]_L$ and $\omega$; and let $f$
be a computable function such that $\phi_{f(e,a,b)}=h_{a_D,b_D}\circ \phi_e$.
But then it is easy to see that the computable function
\[
k(D, e,x)=h^{-1}_{a_D,b_D}(k_{a_D,b_D}(h_{a_D,b_D}[D],f(e,a_D,b_D),x))
\]
witnesses that $L$ is u.f.p.\,: if $a_D<_L b_D$ this follows from the fact
that $k_{a_D,b_D}$ witnesses that the e.i. pre-lattice $L_{a_D,b_D}$ is
u.f.p.\,; if $a_D\equiv_L b_D$ (when $L_{a_D,b_D}$ collapses to one point)
this follows from the fact that $[D]_L=[a_D]_L$, and in this case
$h_{a_D,b_D}^{-1}$ always lands in $[a_D]_L$.
\end{proof}

All the examples seen so far of u.e.i. (or, equivalently, u.f.p.) c.e.
pre-lattices are in fact e.i. pre-lattices, which by the very definition, are
bounded. The following observation answers the natural question of whether
u.e.i.-ness entails boundedness.

\begin{obs}\label{obs:open}
There exist c.e. pre-lattices which are u.e.i. but not bounded.
\end{obs}

\begin{proof}
Fix $L$ to be any e.i. pre-lattice. Let $M$ be $\oplus_{i\in \omega} L$, that
is the direct sum of $\omega$ copies of $L$: $M$ has universe the set of
elements in $L^{\omega}$ with only finitely many non-$0$ entries, and the
operations and pre-ordering relation of $M$ are component-wise; clearly this
universe can be coded as $\omega$. $M$ is not bounded from above, though it
is not hard to show that $M$ is u.e.i. In fact, we claim that for any pair of
u.e.i. c.e. pre-lattices $A$ and $B$, the direct sum $A\oplus B$ is a u.e.i.
c.e. pre-lattice, and an index for a uniform productive function for $A\oplus
B$ can be found uniformly from indices for uniform productive functions for
$A$ and for $B$. Before proving the claim, let us conclude  the argument to
show that $M$ is u.e.i.\,: since every pair of elements from $M$ appear in a
finite sum of copies of $L$, their $\equiv_M$-equivalence classes are e.i.,
and it is uniform to find the index witnessing this, so $M$ is u.e.i. If we
want an example which is neither bounded above nor below, consider $M\oplus
M^\ast$, where $M^\ast$ is the same as $M$ but with the order reversed.

We are only left to show the claim, that if $A$ and $B$ are u.e.i. c.e.
lattices then the c.e. pre-lattice $A\oplus B$, which equals $A\times B$, is
u.e.i. as well. We code $A\oplus B$ as $\omega$ by using the Cantor pairing
function, and let $\pi_0, \pi_1$ be the computable projections associated
with the Cantor pairing function. Since u.e.i. implies u.f.p.
(Fact~\ref{fact:implications}) we know that $A$ and $B$ are u.f.p.: pick two
totalizers $K^A, K^B$ for $\equiv_A$ and $\equiv_B$, respectively. In order
to show that $A\oplus B$ is u.f.p., for every $D,e,x$ let $u^A_{D,e},
u^B_{D,e}$ be indices of partial computable functions as follows: to compute
$\phi_{u^A_{D, e}}(x)$  wait until $\phi_e(x)\downarrow \equiv_{A\oplus B} d$
for some $d\in D$: if and when this happens then make $\phi_{u^A_{D,
e,x}}(0)$ converge to $\pi_0(d)$; otherwise $\phi_{u^A_{D, e,x}}(0)$
diverges. The computation $\phi_{u^B_{D, e,x}}(0)$ is defined similarly,
outputting $\pi_1(d)$ when it converges. Define now
\[
k(D,e,x)=\langle k^A(\pi_0[D],u^A_{D,e}, x),
k^B(\pi_1[D],u^B_{D,e}, x)\rangle.
\]
Using that $K^A, K^B$ are totalizers for $\equiv_A$ and $\equiv_B$,
respectively, it is straightforward to see that $k$ is a totalizer for
$\equiv_{A\oplus B}$. Since (Fact~\ref{fact:implications}) u.f.p. implies
u.e.i. we have that $A\oplus B$ is u.e.i.\,. Uniformity from indices of
uniform productive functions for $A$ and $B$, respectively, to indices of a
productive function for $A\oplus B$ follows from the following facts: the
implications stated in Fact~\ref{fact:implications} are uniform, i.e. one can
uniformly go from an index of a uniform productive function for a u.e.i. ceer
to an index of a totalizer of that ceer, and vice versa; and the proof of the
above claim shows that one can uniformly go from indices of $k^A, k^B$ to
indices of $k$.
\end{proof}

\begin{rem}
Notice that $M$ and $M\oplus M^*$ are distributive if and only if $L$ is
distributive.
\end{rem}

\section{Local universality and uniform density}
When dealing with ordered or pre-ordered structures, one of the most
immediate and important tasks is to study density questions, in particular
whether a given structure is dense or where density locally holds or fails.
For pre-lattices, effective inseparability turns out to be, as far as density
goes, a powerful tool of investigation.

\subsection{Density}
We first look at the case in which a c.e. pre-lattice $L$ need not be bounded
but it has a nontrivial interval $[a,b]_L$ such that the pair $([a]_{L},
[b]_{L})$ is e.i.\,. Before proving in Corollary~\ref{cor:interval} that the
bounded interval $[a,b]_L$ enjoys the universality property already shown of
any e.i. pre-lattice, namely every c.e. pre-order can be computably embedded
into the interval $[a,b]_L$, we first notice, by the following example, that
such a c.e. pre-lattice, even if bounded, need not be u.e.i.\,.

\begin{eg}
Let $L$ be a u.e.i. pre-lattice. Form $L'$ by placing a new element called
$0$ below every element of $L$ and a new element called $1$ above every
element of $L$. Formally, let $h$ be a computable bijection of $\omega$ with
$\omega \smallsetminus \{0,1\}$, and define $\leq_{L'}$ be the c.e.
pre-ordering given by $x \leq_{L'} y$ if and only if
\[
x=0 \lor y=1 \lor [\{x,y\}\cap \{0,1\}=\emptyset \,\&\,
h^{-1}(x) \leq_{L} h^{-1}(y)].
\]
It is now immediate to see that there is a c.e. bounded pre-lattice
$L'=\langle \omega, \wedge, \lor, 0,1,\leq_{L'}\rangle$ such that
$[0]_{L'}=\{0\}$, $[1]_{L'}=\{1\}$, and thus the pair $([0]_{L'},[1]_{L'})$
is not e.i.\,. On the other hand all disjoint pairs $([a]_{L'}, [b]_{L'})$
with $\{a,b\}\cap \{0,1\}= \emptyset$ are e.i.\,.
\end{eg}

\begin{defn}
A c.e. pre-order $R$ is \emph{locally universal} if for every pair $a,b$ such
that $a<_R b$ one can computably embed any c.e. pre-order in the interval
$[a,b]_R$. A c.e. pre-lattice (Boolean pre-algebra, Heyting pre-algebra,
etc.) is \emph{locally universal} if so is its pre-ordering relation.
\end{defn}

\begin{cory}\label{cor:interval}
If $L=\langle \omega, \wedge, \vee, 0,1, \leq_{L} \rangle$ is a c.e.
pre-lattice and $([a]_{L}, [b]_{L})$ are e.i. with $a<_{L}b$ then one can
reduce any c.e. pre-order to the interval $[a,b]_{L}$.
\end{cory}

\begin{proof}
Trivial, since by Lemma~\ref{lem:restriction} $[a,b]_{L}$ is computably
isomorphic to an e.i. pre-lattice.
\end{proof}

\begin{cory}\label{cor:locally-universal}
Every u.e.i. c.e. pre-lattice is locally universal.
\end{cory}

\begin{proof}
This is immediate by Corollary~\ref{cor:interval} since in a u.e.i. c.e.
pre-lattice, all pairs of distinct equivalence classes are e.i.\,.
\end{proof}

\begin{rem}
If $L$ is a u.e.i. c.e. pre-lattice then local universality is uniform: from
$a,b$ and a c.e. index for a c.e. pre-order $R$ one can uniformly find an
index of a  computable function $f$ which is a computable embedding of $R$
into $[a,b]_L$ if $a<_L b$. This follows from the fact that the proof of
Theorem~\ref{thm:ei-lattices-universal} is uniform in an index of $\leq_R$,
and from the observations on uniformity of
Lemmata~\ref{lem:ei-restriction}~and~\ref{lem:restriction} made in the proof
of Theorem~\ref{thm:from-uei-to}.
\end{rem}

\subsection{Uniform density}
Uniform density for c.e. pre-lattices appearing in logic and for c.e.
precomplete pre-lattices is treated in detail in \cite{Shavrukov-Visser}, to
which the reader is referred also for motivations and historical remarks. We
recall the definition.

\begin{defn}
A c.e. pre-ordering $\leq_R$ is \emph{uniformly dense} if there exists a
computable function $f(a,b)$ so that for all $a,b$ if $a<_R b$ then $a <_R
f(a,b) <_R b$, and for all pairs $a',b'$ such that $a \equiv_R a'$ and $b
\equiv_R b'$ we have that $f(a,b) \equiv_R f(a',b')$. A c.e. pre-lattice
(Boolean pre-algebra, Heyting pre-algebra, etc.) is \emph{uniformly dense} if
so is its pre-ordering relation.
\end{defn}

The proof of the next theorem follows, via a natural generalization, along
the lines of the proof of uniform density for c.e. precomplete pre-lattices
in \cite{Shavrukov-Visser}.

\begin{thm}\label{thm:unif-density}
Every u.e.i. c.e. pre-lattice is uniformly dense.
\end{thm}

\begin{proof}
Let $L$ be a u.e.i. pre-lattice, and let $k(D,e,x)$ be a computable function
witnessing that $\equiv_L$ is u.f.p.: see Theorem~\ref{thm:from-uei-to}.

Let $\{e_{a,b}: a,b \in \omega\}$ be a computable list of indices we control
by the Recursion Theorem. Denote by $\equiv^2_L$ the equivalence relation
\[
(a,b) \rel{\equiv^2_L} (a',b')
\Leftrightarrow \left[ a \rel{\equiv_L}
a' \,\&\, b \rel{\equiv_L} b'\right],
\]
and define
\[
f(a,b)=(a \vee j(a,b)) \wedge b,
\]
where $j(a,b)= k(D_{a,b}, e_{a,b}, 0))$ and $D_{a,b}=\{a,b, j(a',b'): \langle
a', b'\rangle < \langle a, b\rangle\}$. (Here and in the following, the
reader is invited to make sure to distinguish between the natural orderings
of $\omega$ for which we use the symbols $\leq, <$, and the pre-ordering or
the strict pre-ordering of $L$ for which we use the symbols $\leq_L, <_L$).

Now we specify how to compute the various $\phi_{e_{a,b}}(0)$.

At step $0$ all computations $\phi_{e_{a,b}}(x)$ are undefined.

At step $s+1$ we consider all pairs $(a,b)$ with $\langle a, b \rangle \le
s$, for which $\phi_{e_{a,b}}(0)$ is still undefined:

\begin{enumerate}
  \item for each such pair $(a,b)$ let $(a_m, b_m) \rel{\equiv^2_L} (a,b)$
      (at $s$) and $\langle a_m, b_m\rangle < \langle a,b\rangle$ is least
      with this property, if there is any such pair. If so, define
      $\phi_{e_{a,b}}(0)=j(a_m,b_m)$. By the properties of the totalizer
      $k$, this makes $j(a_m,b_m)\equiv_L j(a,b)$ and thus
      $f(a_m,b_m)\equiv_L f(a,b)$;

  \item\label{it:2} if after this, $(a,b)$ is still a pair such that
      $\phi_{e_{a,b}}(0)$ is undefined and $a \leq_L b$ at this stage then
      \begin{enumerate}
        \item if  $(a \vee j(a,b)) \wedge b) \equiv_L a$ (at $s$) then
            define $\phi_{e_{a,b}}(0)=b$: this forces $j(a,b) \equiv_L b$
            and thus $a\equiv_L b$;
        \item if  $(a \vee j(a,b)) \wedge b) \equiv_L b$ (at $s$) then
            define $\phi_{e_{a,b}}(0)=a$: this forces $j(a,b) \equiv_L a$
            and thus, again, $a\equiv_L b$.
      \end{enumerate}
\end{enumerate}
This ends the construction. Notice that $\phi_{e_{a,b}}(0)$ is eventually
defined if there is $(a',b')$ in the same $\equiv^2_L$-class as $(a,b)$  and
$\langle a',b' \rangle < \langle a,b \rangle$.

We now show that $f$ is extensional. As $f(a,b) \equiv_L a$  if $a \equiv_L
b$, then extensionality trivially holds when $a \equiv_L b$. Suppose that $a
\cancel{\equiv_L} b$. Then $\phi_{e_{a,b}}(0)$ has not been defined through
either subclause of clause~(\ref{it:2}), as this would give $a \equiv_L b$.
We claim that in this case $f(a,b)\equiv f(a_m,b_m)$ where $(a_m,b_m)$ is in
the same $\equiv^2_L$-class as $(a,b)$ and $\langle a_m, b_m\rangle$ is least
with this property. We prove by induction on $\langle a',b' \rangle$ that
this is true for all $(a',b')$ which is in the same $\equiv^2_L$-class as
$(a_m,b_m)$. The claim is trivial when $(a',b')=(a_m,b_m)$. Otherwise
$\phi_{e_{a,b}}(0)$ is defined at some stage (not through~(\ref{it:2}) as we
have already seen). Thus we have forced $j(a,b) \equiv_L j(a',b')$ with
$(a',b')$ in the same $\equiv^2_L$-class as $(a,b)$ but $\langle a',b'
\rangle < \langle a,b \rangle$. This makes $f(a,b)\equiv_L f(a',b')$. By
inductive assumption, $f(a',b')\equiv_L f(a_m,b_m)$: hence $f(a,b)\equiv_L
f(a_m,b_m)$.

Finally we show that if $a<_L b$ then $a<_L f(a,b) <_L b$: notice that we
have $a\leq_L f(a,b) \leq_L b$ for free as $a \le_L b$. Indeed by
extensionality we may assume that $\langle a, b\rangle$ is least among the
pairs in the $\equiv^2_L$-class of $(a,b)$. If $a<_L f(a,b) <_L b$ does not
hold then sooner or later clause~(\ref{it:2}) would force $a \equiv_L b$, a
contradiction.
\end{proof}

\section{Reducibilities inducing monomorphisms}

Given a class $\mathcal{A}$ of pre-ordered structures of the same type, it is
natural to study reductions (called \emph{$\mathcal{A}$-reductions}) between
the structures in $\mathcal{A}$ which \emph{respect} the operations. For
instance:

\begin{defn}\label{def:homomorphisms}
Let $\mathcal{L}$ be the class of c.e. pre-lattices , and assume that $L_{1}=
\langle \omega, \lor_{1}, \wedge_{1}, \leq_{1}\rangle$ and $L_{2}=\langle
\omega, \lor_{2}, \wedge_{2}, \leq_{2} \rangle$ lie in $\mathcal{L}$. We say
that a computable functions $f$ \emph{$\mathcal{L}$-reduces} $L_{1}$ to
$L_{2}$ (notation: $L_{1} \preceq_\mathcal{L} L_{2}$) if $f$ reduces
$\leq_{L_{1}}$ to $\leq_{L_{2}}$ and for all $x,y \in L_{1}$,
\begin{align*}
f(x\lor_{1} y) &\equiv_{L_{2}} f(x) \lor_{2} f(y),\\
f(x \wedge_{1} y) &\equiv_{L_{2}} f(x) \wedge_{2} f(y).
\end{align*}
Clearly $f$ induces a monomorphism between the associated quotient
structures. Given a class $\mathcal{A}$ of c.e. pre-structures of the same
type, we say that a c.e. pre-structure $L$ is \emph{$\mathcal{A}$-universal}
(or \emph{universal with respect to $\mathcal{A}$}) if $L$ has a type
containing the type of the pre-structures in $\mathcal{A}$ and $A
\preceq_{\mathcal{A}} L^{\textrm{red}}$ for every $A \in \mathcal{A}$ (where
$L^{\textrm{red}}$ is the pre-structure in $\mathcal{A}$ obtained from $L$ by
restricting its type to the type of $\mathcal{A}$).
\end{defn}

In addition to using the symbol $\mathcal{L}$ to denote the class of c.e.
pre-lattices, in the following we will use the following notations:
$\mathcal{B}$ is the class of c.e. Boolean pre-algebras;
$\mathcal{L}^d_{0,1}$ comprises the c.e. distributive bounded pre-lattices;
$\mathcal{L}^d_{0}$ comprises the c.e. distributive pre-lattices with a least
element; $\mathcal{L}^d$ is the class of c.e. distributive pre-lattices. We
will keep the term ``universal'' without specifications, to mean universality
with respect to the c.e. pre-orders: thus a c.e. pre-ordered structure $L$
with c.e. pre-ordering relation $\leq_{L}$ is \emph{universal} if $\leq_R
\preceq \leq_{L}$ for every c.e. pre-order $R$.

\begin{lemma}[\cite{Montagna-Sorbi:Universal}]\label{lem:universal classes}
$\mathcal{L}^d$-universality implies universality;
$\mathcal{L}^d_{0,1}$-universality implies $\mathcal{L}^d$-universality;
$\mathcal{B}$-universality implies $\mathcal{L}^d_{0,1}$-universality.
\end{lemma}

\begin{proof}
The first claim follows from the fact that every c.e. pre-order can be
reduced to the pre-ordering relation of some c.e. distributive pre-lattice,
see for instance the proof of  \cite[Theorem~3.1]{Montagna-Sorbi:Universal}.
The second claim is trivial. For the third claim see
\cite[Theorem~3.2]{Montagna-Sorbi:Universal}.
\end{proof}

\begin{thm} \label{thm:Boolean-universal}\cite{Montagna-Sorbi:Universal}
Every e.i. Boolean pre-algebra is $\mathcal{B}$-universal (by
Lemma~\ref{lem:universal classes}, this implies universality and
$\mathcal{A}$-universality for the classes $\mathcal{A}$ specified in the
Lemma).
\end{thm}

In \cite{Visser1982} Visser introduced a c.e. extension of Heyting
Arithmetic, called $HA^*$ (which is obtained, as described in
\cite{DeJongh-Visser}, as $HA$ plus the Completeness Principle for $HA^*$ by
a fixed point construction). Let $\HA^*$ be the Heyting pre-algebra of the
Lindenbaum sentence algebra of $HA^*$ (i.e $\HA^*=\textrm{i}(HA^*)$ in the
notation of Section~\ref{special-HA})): clearly $\HA^*$ is an e.i. Heyting
pre-algebra as $HA^*$ extends $HA$.

Building on some proof techniques due to
Shavrukov~\cite{Shavrukov-Subalgebras}, Visser~\cite{Visser1985}, and
Zambella~\cite{ZambellaNJFL}, De~Jongh and Visser~\cite{DeJongh-Visser} prove
the following $\mathcal{A}$-universality result, where in this case
$\mathcal{A}$ is the class of all c.e. prime Heyting pre-algebras (a Heyting
algebra is \emph{prime} if its greatest element is join-irreducible).

\begin{thm}[\cite{DeJongh-Visser}]
$\HA^*$ is $\mathcal{A}$-universal.
\end{thm}

\subsection{An e.i. pre-lattice which not $\mathcal{L}^d_{0,1}$-universal}
Although every e.i. pre-lattice is universal, and despite the fact that the
e.i. Boolean pre-algebras are $\mathcal{B}$-universal by
Theorem~\ref{thm:Boolean-universal}, and also the fact there exist e.i.
distributive pre-lattices which are $\mathcal{L}^d_{0,1}$-universal as we
will observe in Section~\ref{sct:applications}, the following theorem shows
that one should not be led to conclude that all e.i. pre-lattices are
$\mathcal{L}$-universal, or even that all e.i. distributive pre-lattices are
$\mathcal{L}^d$-universal.

\begin{thm}\label{thm:not-all-of-them}
There exist e.i. distributive pre-lattices that are not
$\mathcal{L}^d$-universal.
\end{thm}

\begin{proof}
Since we know from Example~\ref{ex:eidistr} that there is an e.i.
distributive pre-lattice $L^{d}_{01}$ which is a free bounded distributive
pre-lattice on a countably infinite set of generators, it is enough to show
that for every $X$ not all countable distributive lattices which are
isomorphic to quotient structures of c.e. pre-lattices (see
Definition~\ref{defn:associated-quotient}) can be lattice-embedded into
$F_{0,1}^d(X)$: the claim then follows from the fact that by
Example~\ref{ex:eidistr} the quotient structure corresponding to $L^{d}_{01}$
is isomorphic to some such $F_{0,1}^d(X)$. To see this, take $B$ to be any
infinite Boolean algebra. Suppose that we have a lattice theoretic-theoretic
embedding of $B$ into $F_{0,1}^d(X)$. Clearly $0 \in B$ can not be mapped to
$0\in F_{0,1}^d(X)$, as the former $0$ is meet-reducible whereas the latter
one is not. So the least element $0 \in B$ must be mapped to an element $a\in
F^d(X)$ (recall that $F^d_{0,1}(X)=0 \oplus F^d(X) \oplus 1$ where $F^d(X)$
is the free distributive lattice on $X$, see \cite{Balbes-Dwinger}).
Likewise, the greatest element $01\in B$  must be mapped to an element $b\in
F^d(X)$. Now each element $x\in F^d(X)$ can be written (in a unique way,
called \emph{normal form} of $x$) as $\sum_{i\in I_x} U_i^x$ where $I_x$ is a
nonempty finite set, $\sum$ denotes join, each $U_i^x$ is a nonempty finite
set of distinct generators, and for simplicity we identify (as in clause
notation) such a finite set with the meet of its element (i.e., to be
explicit, $x=\sum_{i\in I_x} \prod U_i^x$); it can be also assumed that the
$U_i^x$ are all $\subseteq$-incomparable (if $U_i^x \subseteq U_j^x$, just
delete $j$ from $I_x$); moreover given $x=\sum_{i\in I_x} U_i^x$ and
$y=\sum_{i\in I_y} U_i^y$, we have
\[
x \le y \Leftrightarrow (\forall i\in I_x)(\exists j \in I_y)[U^y_j
\subseteq U^x_i].
\]
Now suppose that $a=x\wedge y$, and $b=x\vee y$: thus (representing $a,b,x,y$
in normal form) we have
\begin{align*}
\sum_{i\in I_a} U_i^a&=\sum_{(u,v)\in I_x\times I_y} (U^x_u\cup U^y_v),\\
\sum_{i\in I_b} U_i^b&=\sum_{u\in I_x} U^x_u + \sum_{v\in I_y}  U^y_v.
\end{align*}
By selecting the pairs $(u,v)$ for which $U^x_u\cup U^y_v$ is minimal under
inclusion (so that its join is maximal) we can pick a subset $J \subseteq
I_x\times I_y$ such that
\begin{equation}\label{eqn:1}
\sum_{i\in I_a} U_i^a=\sum_{(u,v)\in J} (U^x_u\cup U^y_v).
\end{equation}
and all the $U^x_u\cup U^y_v$, with $(u,v)\in J$, are
$\subseteq$-incomparable. Using $\ge$ of (\ref{eqn:1}) it follows that
\[
(\forall (u,v)\in J)(\exists i\in I_a)[U^a_i \subseteq U^x_u
\cup U^y_v];
\]
on the other hand, using $\le$ of (\ref{eqn:1}) we have
\[
(\forall i\in I_a) (\exists (u,v)\in J)[U^x_u \cup U^y_v
\subseteq U^a_i].
\]
It follows that
\[
(\forall (u,v)\in J)(\exists i\in I_a)(\exists (u',v')\in
J)[U^x_{u'} \cup U^y_{v'} \subseteq U^a_i \subseteq U^x_u \cup U^y_v]:
\]
by incomparability, for all $(u,v), (u',v') \in J$ it follows that  $U^x_{u'}
\cup U^y_{v'}=U^x_u \cup U^y_v$, and thus $U_u^x$ and $U_v^y$ are subsets of
$U=\bigcup_{i \in I_a} U^a_i$. Consider now the equality $b=\sum_{u\in I_x}
U^x_u + \sum_{v\in I_y} U^y_v$ and let us pick subsets $H \subseteq I_x$,
$K\subseteq I_y$ which select the sets $U^x_u$ and $V^y_v$ which are
$\subseteq$-minimal and $b=\sum_{u\in H} U^x_u + \sum_{v\in K} U^y_v$. By
arguing as before we can show that for every $u\in H$ and $k \in K$ we have
that $U^x_h, Y^y_k \subseteq V=\bigcup_{i \in I_b} U^b_i$. Finally, suppose
that some $U^x_i$ or some $U^y_j$ has been left out of the reduced family
that joins to $b$ (i.e. in the former case $i \notin H$ and the latter case
$j \notin K$): assume for instance that $i_0\in I_x\smallsetminus H$, for
some $i_0$. Then (by $\subseteq$-incomparability the sets $U^x_i$) there is
$j_0 \in I_y$ such that $V^y_{j_0}\subseteq U^x_{i_0}$; but then for the pair
$(i_0,j_0)$ we have $U^x_{i_0} \cup V^y_{j_0}= U^x_{i_0}$. If $(i_0,j_0) \in
J$ then we know that $U^x_{i_0} \subseteq U$; on the other hand if $(i_0,j_0)
\notin J$ then there exists a pair $(i_1,j_1)\in J$ such that $U^x_{i_1} \cup
U^y_{j_1} \subseteq U^x_{i_0}$, which by incomparability implies $i_0=i_1$
and thus we have that $U^x_{i_0} \subseteq U$, as $U^x_{i_1} \subseteq U$. If
instead for some $j_0$ we have $j_0\in I_y\smallsetminus K$ so that
$U^y_{j_0}$ has been left out of the reduced family that joins to $b$ then by
a similar argument we conclude that $U^y_{j_0} \subseteq U$. In conclusion,
for all $(u,v)\in I_x\times I_y$ we have that  we have $U^x_u \cup U^y_v
\subseteq U\cup V$. Thus there can be finitely only many pairs $x,y$ that
meet to $a$ and join to $b$, since the normal forms of $x,y$ may involve only
generators in $U\cup V$. But being infinite, the Boolean algebra $B$ has
infinitely many distinct pairs of elements that meet to $0$ and join to $1$.
Thus the images under the embedding of these pairs would provide an infinite
set of pairs of elements that in $F^d_{0,1}(X)$ meet to $a$ and join to $b$,
contradiction.
\end{proof}

\begin{rem}
Things are even more dramatic if we consider non-distributive pre-lattices.
Galvin and J\`onsson~\cite{Galvin-Jonsson-Canadian} prove that a distributive
lattice $L$ can be embedded in a free lattice if and only if $L$ is a
countable linear sum of lattices where: each lattice in the linear sum is
either a one element lattice, or an eight element Boolean algebra $2 \times 2
\times 2$, or a direct product of the two element chain with a countable
chain.

Therefore no c.e. distributive pre-lattice with quotient structure isomorphic
to $0 \oplus L \oplus 1$, where $L$ is not isomorphic with any of the
lattices characterized by Galvin and J\`onsson as above, can be reduced to
$L^{nd}_{01}$ (whose quotient structure is isomorphic to a lattice of the
form $0 \oplus F^{nd}(X) \oplus 1$ where $F^{nd}(X)$ is a free lattice) of
Example~\ref{ex:nondistr} preserving $\wedge$ and $\lor$.
\end{rem}

\section{Some applications}\label{sct:applications}
In this section we give some applications of
Theorem~\ref{thm:ei-lattices-universal} and Theorem~\ref{thm:unif-density}.
Our results on universality, local universality, and uniform density hold of
\emph{all} e.i. pre-lattices, which form a large enough class to include some
c.e. pre-structures of particular interest which have been widely studied and
for which some or all of these results have been proved already, and also
some interesting cases (listed in Section~\ref{ssct:new-applications}) where
these were not known. In Section~\ref{ssct:new-applications} we list some of
these new cases. In Section~\ref{ssct:reviewing} we will briefly review the
literature on what was already known about e.i. pre-structures, and
universality with respect to c.e. pre-orders, or even
$\mathcal{A}$-universality for various classes $\mathcal{A}$.

\subsection{New applications}\label{ssct:new-applications}
In the next three corollaries we point out three applications of our theorems
which, in their generality,  seem to have gone unnoticed so far.

\begin{cory}
If $L$ is a c.e. precomplete pre-lattice then $L$ is locally universal, i.e.\
in any non-trivial interval $[a,b]_L$ one can embed every c.e. pre-order.
\end{cory}

\begin{proof}
By Corollary~\ref{cor:interval}.
\end{proof}

\begin{cory}\label{cor:sigman}
If $T$ is a consistent c.e. extension of $R$ or $Q$ then for every $n \ge 1$
the pre-lattice of sentences $L_{\Sigma_n/T}$ satisfies:
\begin{enumerate}
 \item  $L_{\Sigma_n/T}$ is locally universal;
      \item $L_{\Sigma_n/T}$ is uniformly dense.
\end{enumerate}
\end{cory}

\begin{proof}
Use Theorem~\ref{thm:Qei}: then Claim (1) comes from
Corollary~\ref{cor:interval}. Claim (2) comes from
Theorem~\ref{thm:unif-density}.
\end{proof}

\begin{cory}\label{cor:HA}
If $iT$ is a c.e. consistent intuitionistic theory extending $HA$ then the
Heyting pre-algebra $\IT$ of the Lindenbaum sentence algebra of $iT$
satisfies:
\begin{enumerate}
  \item $\IT$ is locally universal;
  \item $\IT$ is uniformly dense.
\end{enumerate}
\end{cory}

\begin{proof}
Claim (1) comes Corollary~\ref{cor:interval}. Claim (2) comes from
Theorem~\ref{thm:unif-density}.
\end{proof}

\subsection{Reviewing (some of) the existing literature}\label{ssct:reviewing}
For e.i. Boolean pre-algebras, embedding results via computable functions
yielding universality, were already known in the literature. We have already
seen (Theorem~\ref{thm:Boolean-universal}) that each e.i. Boolean pre-algebra
is $\mathcal{B}$-universal. This theorem alone is enough to establish local
universality of the e.i. Boolean pre-algebras: in fact there is no need to
directly use Corollary~\ref{cor:interval} as any non-trivial interval of an
e.i. Boolean pre-algebra is itself computably isomorphic (use
Lemma~\ref{lem:restriction} and the fact that any such interval is a c.e.
Boolean pre-algebra) to an e.i. Boolean pre-algebra, and thus any such
interval is $\mathcal{B}$-universal. Moreover, uniform density holds in any
e.i. Boolean pre-algebra, as all e.i. Boolean pre-algebras are computably
isomorphic (see Theorem~\ref{thm:Boolean-isomorphic}) and Shavrukov and
Visser~\cite{Shavrukov-Visser} show that uniform density holds in any $B_T$
of Example~\ref{ex:BT}.

Let us now move to c.e. precomplete pre-lattices and pre-lattices of
sentences. Uniform density for c.e. precomplete pre-lattices was proved by
Shavrukov and Visser in \cite{Shavrukov-Visser}. Now, if $T$ is a c.e.
consistent extension of elementary arithmetic $EA$ then by
Theorem~\ref{thm:EA-precomplete} $L_{\Sigma_n/T}$ is precomplete, and thus
uniformly dense. Moreover, it was already known that for such a $T$, and for
$n\ge 1$, $L_{\Sigma_n/T}$ is locally $\mathcal{L}^d_{0}$-universal, and
locally $\mathcal{L}^d_{0,1}$-universal if $n \ge 2$. Indeed
$\mathcal{L}^d_{0}$-universality of any $L_{\Sigma_n/T}$ for $n \ge 1$, and
$\mathcal{L}^d_{01}$-universality of any $L_{\Sigma_n/T}$ for $n \ge 2$ was
first proved in \cite{Montagna-Sorbi:Universal} where it is shown that the
reducing functions can in fact be taken of the form $x\mapsto
\gamma(\psi(\ol{x}))$ with $\psi(v)$ a $\Sigma_n$ formula, $\ol{x}$ the
numeral term corresponding to $x$, and $\gamma(\psi(\ol{x}))$ the G\"odel
number of $\psi(\ol{x})$ (under the $\gamma$ exploited in building the
pre-lattice of sentences).

Shavrukov~\cite{Shavrukov} proved that for $n\ge 1$ the c.e. Boolean
pre-algebra $B_{\Delta_{n}/T}$ (having as domain the $\Delta_n$-sentences,
and regarded as a c.e. Boolean pre-algebra by Remark~\ref{rem:more-general})
is an e.i. Boolean pre-algebra unless $n=1$ and $T$ is $\Sigma_{1}$-sound;
and if $n \ge 1$, $\alpha, \beta$ are $\Sigma_n$ sentences, $\vdash_{T}
\alpha \rightarrow \beta$, and $\not \vdash_{T} \beta \rightarrow \alpha$,
then $[\alpha, \beta]_{\Delta_{n}/T}$ (i.e. the $\Delta^0_n$ sentences in the
interval) is an e.i. Boolean pre-algebra unless $n=1$, $T$ is
$\Sigma_{1}$-sound, and $T\vdash \beta$. $\mathcal{A}$-universality of
$L_{\Sigma_n/T}$ (for the relevant class $\mathcal{A}$) finally follows by
$\mathcal{B}$-universality of the e.i. Boolean pre-algebras
(Theorem~\ref{thm:Boolean-universal}) and Lemma~\ref{lem:universal classes}.

\section{Computable isomorphism types of universal c.e. pre-orders}
It would be interesting to characterize ``natural'' classes of universal
pre-structures falling in a single computable isomorphism type, as happens
for instance for the e.i. Boolean pre-algebras:

\begin{thm}[(\cite{Montagna-Sorbi:Universal, Nerode-Remmel:Survey}, based on
	\cite{Pour-El-Kripke})] \label{thm:Boolean-isomorphic}
All e.i. Boolean pre-algebras are computably isomorphic.
\end{thm}

Some natural computable isomorphism types have been pointed out for
universal ceers. We recall that a \emph{diagonal} function for an equivalence
relation $R$ is a computable function $\delta$ such that $x \cancel{\rel{R}}
\delta(x)$ for all $x$. The following are known: All u.f.p. ceers endowed
with a diagonal function are computably isomorphic, see \cite{Montagna:ufp}
(and independently  \cite{Lachlan}); all precomplete ceers are computably
isomorphic, see \cite{Lachlan}.

Nonetheless, there are several interesting different computable isomorphism
types of universal c.e. pre-orderings, even if we restrict attention to the
pre-orderings relative to e.i. (hence bounded by definition) pre-structures.

\begin{thm}\label{thm:several}
The following list provides classes of e.i. pre-structures such that for each
pair $(\mathcal{A}, \mathcal{B})$ of distinct classes (i.e., (1), (2a), (2b)
(2c), (3))) in the list there is a pair of objects $(A,B) \in (\mathcal{A}
\times\mathcal{B})$, with $A,B$ not computably isomorphic.

\begin{enumerate}
  \item e.i. Boolean pre-algebras;
  \item e.i. distributive pre-lattices; inside this class, one can further
      list the following subclasses:
  \begin{enumerate}
    \item\label{it:Heyting} e.i. Heyting pre-algebras;
    \item\label{it:free} e.i. free distributive pre-lattices on a
        countably infinite set of generators;
    \item\label{it:precompl-latt-sentences} c.e. precomplete distributive
        pre-lattices, including pre-lattices of sentences of consistent
        extensions of Robinson's $Q$ or $R$;
  \end{enumerate}
  \item e.i. non-distributive pre-lattices.
\end{enumerate}
\end{thm}

\begin{proof}
Of course distinctions between some of the classes derive immediately from
evident structural differences (for instance a free distributive lattice can
not be isomorphic with a Boolean algebra, and so on) which immediately imply
that for a given pair $(A,B) \in (\mathcal{A} \times\mathcal{B})$ of
pre-structures no isomorphism may exist  between the corresponding quotient
structures, and thus no computable isomorphism may exist between the
pre-structures. Apart from these obvious considerations, the proof follows
from previous results proved in this paper or recalled in this paper from the
literature. We list here some of the distinguishing properties. The c.e.
precomplete pre-lattices do not have a diagonal function (this is a
consequence of the celebrated Ershov Fixed Point Theorem for precomplete
equivalence relations). If $L$ is the Heyting pre-algebra of the Lindenbaum
sentence algebra of $HA$ then $L$ is e.i. and has a diagonal function, so no
computably isomorphic copy of it can lie in \ref{it:precompl-latt-sentences},
and can not lie in \ref{it:free} either since in every free bounded
distributive lattice the least element is meet-irreducible, but this does not
hold in the Lindenbaum sentence algebra of $HA$. For the same reason, if $n
\ge 1$ and $T$ is a c.e. consistent extension of $EA$ then no isomorphic copy
of $L_{\Sigma_n/T}$ can lie in (2b).
\end{proof}

One must not think that each of the classes displayed in
Theorem~\ref{thm:several} contains only one computable isomorphism class. For
instance, as already remarked, if $T$ is $\Sigma_1$-sound then the
precomplete distributive pre-lattice $L_{\Sigma_1/T}$  is not isomorphic to
any $L_{\Sigma_n/T}$, for $n \ge 2$.

E.i. Boolean pre-algebras are free on infinitely many generators, since so is
$B_T$ of Example~\ref{ex:BT}. However, even adding freeness, the analogue of
Theorem~\ref{thm:Boolean-isomorphic} for e.i. distributive pre-lattices
fails, as shown in next theorem.

\begin{thm}\label{thm:ei-prelattices-not-isomorphic}
There are e.i. distributive pre-lattices that are free on countably
infinitely many generators but are not computably isomorphic.
\end{thm}

\begin{proof}
We construct two e.i. bounded distributive pre-lattices $L_1, L_2$, each of
which is free on a countably infinite sets of generators, and such that for
every $i$, the \emph{requirement $R_i$} is satisfied, i.e. the partial
computable function $\phi_i$ does not induce an isomorphism between the
quotient structures associated with $L_1$ and $L_2$. We take $L_2$ to be the
pre-lattice $L^d_{01}$ of Example~\ref{ex:eidistr}, but this time we show how
to build this pre-lattice and the congruence $\alpha$ via a construction in
stages, which we then synchronize with (part of) the construction of $L_1$.

As in the discussion in Section~\ref{sssct:eidistr} consider a computable
presentation $F^d_{0,1}(X)=\langle \omega, \wedge, \lor, 0,1\rangle$ of the
free bounded distributive lattice on a decidable infinite set $X$ (with $0,1
\notin X$) which is computably listed without repetitions by $\{x_i: i \in
\omega\}$, and let $(U,V)$ an e.i. pair of c.e. sets.

We start with specifying a construction of $L_2=L^d_{01}$ in stages. Assume
that $\{U_s: s\in \omega\}, \{V_s: s\in \omega\}$ are computable
approximations to $U,V$ respectively, i.e. computable sequences of finite
sets giving $U=\bigcup_s U_s$ and $V=\bigcup_s V_s$. Let $X_s=\{x_i: i <s\}$,
and let $\alpha_{s}$ be the congruence on  $F^d_{01}(X_s)$ generated by the
pairs $\{(x_i,0): x_i \in X_s \,\&\, i \in U_s\} \cup \{(x_i,1): x_i \in X_s
\,\&\, i \in V_s\}$. It is easy to see (First Isomorphism Theorem) that
${F^d_{01}(X_s)}_{/\alpha_s} \simeq F^d_{01}(X_s\smallsetminus \{x_i: i \in
U_s\cup V_s\})$, and if $\alpha$ is the congruence on $F^d_{01}(X)$ exhibited
in Section~\ref{sssct:eidistr} then $\alpha=\bigcup_s \alpha_{s}$. Taking
$L_2=L^d_{01}$ of Example~\ref{ex:eidistr} we know that $L_2$ is e.i.\,,
${L_2}_{/\alpha}={F^d_{01}(X)}_{/\alpha}\simeq F^d_{01}(X\smallsetminus
\{x_i: i\in U\cup V\})$, and $L_2$ is free on a countably infinite set of
generators. We use the sequence $\{{F^d_{01}(X_s)}_{/\alpha_s}: s \in
\omega\}$ to approximate $F^d_{01}(X)_{/\alpha}$ in the construction of
$L_1$.

We now specify how to build $L_1$. Split $X$ as $X=Y\cup Z$, with $Y,Z$
decidable and infinite. Split also $Z$ into two infinite decidable sets $Z_1,
Z_2$ so that $Z_1$ is computably enumerated without repetitions by $\{z_i^1:
i\in \omega\}$.

We define in stages sequences $\{X_{1,s}: s\in \omega\}$, $\{H_s: s\in
\omega\}$, $\{\beta_s: s\in \omega\}$, where $H_s \subseteq X_{1,s}$,
$X_{1,s}$ is a finite subset of $X$, $\beta_s$ is a congruence on
$F^d_{0,1}(X_{1,s})$; moreover let $G_s=X_{1,s}\smallsetminus H_s$. Recall
that in a bounded lattice an element $x$ is meet-irreducible if $x \ne 0,1$
and is not the meet of any two strictly bigger elements.

\medskip

\paragraph{Stage $0$}
Let $X_{1,0}=H_0=\emptyset$ so that $F^d_{0,1}(X_{1,0})$ is the two-element
bounded lattice, and let $\beta_0$ be the identity equivalence relation on
$\{0,1\}$.
\medskip

\paragraph{Stage $s+1$, say $s=\langle i, t\rangle$}
If $R_i$ has already acted, or $z^1_i \notin X_{1,s}$, or $\phi_i(z^1_i)$ has
not as yet converged in fewer than $s$ steps to an $x \in F^d_{0,1}(X_s)$, or
$\phi_i(z^1_i)$ has already converged to an $x \in F^d_{0,1}(X_s)$, such that
$[x]_{\alpha_{s}}$ is meet-reducible in ${F^d_{0,1}(X_s)}_{/\alpha_{s}}$,
then let $X_{1,s+1}=X_{1,s} \cup \{x_s\}$, $H_{s+1}=H_s$, and let
$\beta_{s+1}$ be the congruence on $F^d_{0,1}(X_{1,s+1})$ generated by the
pairs in $\beta_s$. Otherwise, $R_i$ \emph{acts} by choosing the least two
distinct elements $z^2_{k_1}, z^2_{k_2}\in Z_2$ not as yet used at any stage
of the construction; let $X_{1,s+1}= X_{1,s} \cup \{x_s, z^2_{k_1},
z^2_{k_2}\}$, let $\beta_{s+1}$ be the congruence on $F^d_{0,1}(X_{1,s+1})$
generated by $\beta_{s}$ plus the pair $(z^1_i,z^2_{k_1} \wedge z^2_{k_2})$;
let $H_{s+1}=H_{s} \cup \{z^1_j\}$.

\medskip

The following can be easily checked. If $H=\bigcup_s H_s$ and
$G=X\smallsetminus H$ then $G=\{g:(\exists t)(\forall s \ge t)[g \in G_s]\}$
and is comprised of all elements $g \in X$ such that $g$ is not collapsed at
some stage to the meet of two elements of $Z_2$ by the c.e. congruence
$\beta$ on $F^d_{0,1}(X)$ where $\beta=\bigcup_s \beta_s$; moreover $Y
\subseteq G$ as $Y\cap H=\emptyset$. Let now $\alpha_1$ be the c.e.
congruence on $F^d_{0,1}(X)$ generated by the pairs in $\beta \cup \{(y_i,0):
i \in U\} \cup \{(y_i,1): i \in V\}$.  On numbers $x,y$ define $x \leq_{L_1}
y$ if $[x]_{\alpha_1}= [x \wedge y]_{\alpha_1}$. Then $L_1=\langle \omega,
\wedge, \lor, 0,1, \leq_{L_1}\rangle$, where $\wedge, \lor$ are the same
operations as in $F^d_{0,1}(X)$, is a c.e. pre-lattice whose associated
quotient structure is $F^d_{0,1}(X)_{/\alpha_1}$. By an argument similar to
the one used for $\alpha$ and $L^d_{01}$ in Section~\ref{sssct:eidistr}, one
can show that $L_1$ is e.i. and ${F^d_{0,1}(X)}_{/\alpha_1} \simeq
F^d_{0,1}(G\smallsetminus \{y_i: i\in U\cup V\})$ so that $L_1$ is free on a
countably infinite set of generators.

It remains to show that for every $i$, $\phi_i$ does not induce an
isomorphism from ${L_1}_{/\alpha_1}(=F^d_{0,1}(X)_{/\alpha_1})$ to
${L_2}_{/\alpha}(=F^d_{0,1}(X)_{/\alpha})$. Assume $\phi_i(z^1_i)$ converges
and $\phi_i(z_i^1)=x$: if $[x]_{\alpha}\in \{[0]_{\alpha}, [1]_{\alpha}\}$ or
$[x]_{\alpha}$ is meet-irreducible then at some stage in the construction of
$\beta$  we introduce a pair $z^2_{k_1}, z^2_{k_2}$, and make
$[z^1_i]_{\alpha_1}$ meet-reducible since we make $[z^1_i]_{\alpha_1}=
[z^2_{k_1}]_{\alpha_1} \wedge [z^2_{k_2}]_{\alpha_1}$, and on the other hand
$[z^2_{k_1}]_{\alpha_1}, [z^2_{k_2}]_{\alpha_1}$ are incomparable as under
the isomorphism ${F^d_{0,1}}(X)_{/\alpha_1} \simeq F^d_{0,1}(G\smallsetminus
\{y_i: i \in U\cup V\})$ they correspond to two distinct generators of
$F^d_{0,1}(G\smallsetminus \{y_i: i \in U\cup V\})$. Similarly if
$[x]_{\alpha}$ is meet-reducible then for every $s$ with $x \in
F^d_{01}(X_s)$ we have that $[x]_{\alpha_{s}}$ is meet-reducible, thus
$[z^1_i]_{\alpha_1}$ is meet-irreducible in ${L_1}_{/\alpha_1}$ as
$[z^1_i]_{\alpha_1}$ corresponds (under the isomorphism ) to a generator of
$F^d_{0,1}(G\smallsetminus \{y_i: i \in U\cup V\})$ since $z_i^1$ is never
extracted from $G$. This shows that $\phi_i$ does not induce an isomorphism,
since meet-reducibility is a property that is invariant under isomorphisms.
\end{proof}

\section{Acknowledgements}
The authors wish to thank (in alphabetical order) Professors Bekhlemishev,
Kabylzhanova, Nies, Pianigiani, and Visser for their helpful comments and
suggestions during the preparation of this paper. Special thanks are due to
Professor Shavrukov for having carefully read a first draft of the paper, and
for his many corrections and suggestions which have greatly contributed to
improve the presentation of the paper.


\end{document}